\documentclass[11pt]{amsart}

\usepackage{epsf}
\usepackage{graphicx}
\usepackage{harvard}
\usepackage{amssymb}
\usepackage{amsmath}
\usepackage{amsthm}
\usepackage{amsopn}
\usepackage{mathrsfs}
\usepackage{a4}

\numberwithin{equation}{section}

\theoremstyle{plain}                    
\newtheorem{thm}{Theorem}[section]

\newtheorem{lem}[thm]{Lemma}
\newtheorem{prop}[thm]{Proposition}

\theoremstyle{definition}

\theoremstyle{remark}
\newtheorem{rem}[thm]{Remark}      

\swapnumbers
\newtheorem{ack}{Acknowledgment}        

\def\R{{\mathbb R}}
\def\N{{\mathbb N}}
\def\C{{\mathbb C}}
\def\E{{\mathbb E}}
\def\Lip{\mbox{\rm Lip}}
\def\1{\mbox{I\hspace{-.6em}1}}
\def\FT{{\mathcal F}}

\newcommand{\be}{\begin{equation}}
\newcommand{\bea}{\begin{eqnarray}}
\newcommand{\eea}{\end{eqnarray}}
\newcommand{\bean}{\begin{eqnarray*}}
\newcommand{\eean}{\end{eqnarray*}}

\DeclareMathOperator{\Var}{Var} 
 
\DeclareMathOperator{\Exp}{Exp}

\providecommand{\eps}{\varepsilon}
\renewcommand{\phi}{\varphi}
\renewcommand{\theta}{\vartheta}

\renewcommand{\cdot}{{\scriptstyle \bullet} }
\providecommand{\abs}[1]{\lvert #1 \rvert}
\providecommand{\norm}[1]{\lVert #1 \rVert}

\providecommand{\babs}[1]{{\Bigl\lvert #1 \Bigr\rvert}}

\renewcommand{\Re}{\operatorname{Re}}
\renewcommand{\Im}{\operatorname{Im}}
\renewcommand{\le}{\leqslant}\renewcommand{\leq}{\leqslant}
\renewcommand{\ge}{\geqslant}\renewcommand{\geq}{\geqslant}

\newcommand{\cit}[1]{\citeasnoun{#1}}

\begin{document}

\thispagestyle{empty}

\baselineskip14pt

\vspace*{0.6cm}

\begin{center}
{\large \sc Nonparametric estimation for L\'evy processes\\ from low-frequency
observations}
\end{center}

\vspace*{0.6cm}

\begin{center}
Michael H.~Neumann  \\
\noindent Friedrich-Schiller-Universit\"at Jena \\
Institut f\"ur Stochastik\\
Ernst-Abbe-Platz 2\\
D -- 07743 Jena, \ Germany\\
E-mail: mneumann@mathematik.uni-jena.de\\[0.4cm]
Markus Rei{\ss}\\
Ruprecht-Karls-Universit\"at Heidelberg\\
Institut f\"ur Angewandte Mathematik\\
Im Neuenheimer Feld 294\\
D -- 69120 Heidelberg, \ Germany\\
E-mail: reiss@statlab.uni-heidelberg.de\\[1.5cm]
\end{center}

\vspace*{-1.0cm}

\begin{center}
{\bf Abstract}
\end{center}
\begin{quote}
We suppose that a L\'evy process is observed at discrete time
points. A rather general construction of minimum-distance estimators
is shown to give consistent estimators of the L\'evy-Khinchine
characteristics as the number of observations tends to infinity,
keeping the observation distance fixed. For a specific
$C^2$-criterion this estimator is rate-optimal. The connection with
deconvolution and inverse problems is explained. A key step in the
proof is a uniform control on the deviations of the empirical
characteristic function on the whole real line.
\end{quote}

\vspace*{3cm}

\footnoterule \noindent {\sl 2000 Mathematics Subject
Classification}. Primary 62G15;
secondary 62M15.
{\sl Short title}. Nonparametric estimation for L\'evy processes.

\newpage

\renewcommand{\baselinestretch}{1.2}
\small\normalsize

\pagestyle{headings} \setcounter{page}{1}

\section{Introduction}
\label{S1}

L\'evy processes form the fundamental building block for stochastic
continuous-time models with jumps. There is an important trend using
L\'evy models in finance, see \cit{CT04}, but also many recent
models in physics or biology rely on L\'evy processes. We consider
here the problem of estimating the L\'evy-Khintchine characteristics
from time-discrete observations of a L\'evy process. Since these
characteristics involve the L\'evy measure (or jump measure) and we
do not want to impose a parametric model, we face a nonparametric
estimation problem.

When the L\'evy process $(X_t)_{t\geq 0}$ is observed at high frequency, at
times $(t_i)_{i=0,\ldots,n}$ with $\max_i(t_i-t_{i-1})$ small, then
a large increment $X_{t_i}-X_{t_{i-1}}$ indicates that a jump
occurred between time $t_{i-1}$ and $t_i$. Based on this insight and
the continuous-time observation analogue, nonparametric inference
for L\'evy processes from high-frequency data has been considered by
\cit{BB82}, \cit{FH06} and \cit{Nish07}. For low-frequency
observations, however, we cannot be sure to what extent the
increment $X_{t_i}-X_{t_{i-1}}$ is due to one or several jumps or
just to the Brownian motion part of the L\'evy process. The only way
to draw inference is to use that the increments form independent
realisations of infinitely divisible probability distributions. We
shall assume that we dispose of equidistant observations at
$t_i=i\Delta$, $i=0,\ldots,n$, and consider the asymptotic behaviour
of estimators for $n\to\infty$ and $\Delta>0$ fixed. This can be
cast into the classical framework of i.i.d.~observations
$(X_{i\Delta}-X_{(i-1)\Delta})_{i=1,\ldots,n}$ from an infinitely
divisible distribution. A natural question in this framework is to
estimate the underlying L\'evy-Khintchine characteristics. In this
general setting we are only aware of the work by \cit{WK03} who propose
and implement an approach for estimating the jump distribution by a
fixed spectral cut-off procedure, which is related to the pilot
estimator in Section \ref{S4bis} below. In the special case of
compound Poisson processes the problem of estimating the jump
density is known as decompounding, see \cit{EGS07}, \cit{Gug07} and
the references therein. For parametric inference under the
assumption of a stable law see e.g. \cit{FM81b}. A related
low-frequency problem for the canonical function in
L\'evy-Ornstein-Uhlenbeck processes has been considered by
\cit{Hol05}, where a consistent estimator has been constructed.

In Section \ref{S2} we recall basic facts about L\'evy processes and
prepare the idea of minimum-distance estimators based on the
empirical characteristic function. Under very general conditions we
then show in Section \ref{S3} consistency of these estimators for
the L\'evy-Khintchine characteristics. The only way to achieve this
is to merge the diffusion coefficient~$\sigma^2$ and the L\'evy
measure $\nu$ to a single quantity $\nu_\sigma$, which is a finite
Borel measure, and to consider weak convergence of estimators of
$\nu_\sigma$. In Section \ref{S4} we construct a rate-optimal
estimator using a minimum-distance fit, based on a $C^2$-criterion
for the empirical characteristic function. A fundamental tool is
Theorem \ref{T4.1}, which gives a uniform control on the deviations
of the empirical characteristic function on the whole real line and
may be of independent interest. The optimal rates of convergence
depend on the decay of the characteristic function as in
deconvolution problems. Interestingly, our estimator attains the
optimal rates without knowing this decay behaviour and without any
further regularisation parameter. In Section \ref{S4bis} we briefly
discuss the implementation of the estimator, using a two-step
procedure, and show a typical numerical example. Most proofs are
postponed to Section \ref{S5}.

\section{Basic notions, assumptions, and a few simple facts}
\label{S2}

We assume that we observe a one-dimensional L\'evy process
$(X_t)_{t\geq 0}$ at equidistant time points $0=t_0<t_1<\cdots
<t_n$. Such a process is characterized by its characteristic
function
\[
\phi(u,t;\bar b,\sigma,\nu) \,:=\, \E[\exp(iu X_t)] \,=\, \exp(t \; \Psi(u;\bar
b,\sigma,\nu)), \qquad u\in\R,
\]
where
\[
\Psi(u) \,=\, \Psi(u;\bar b,\sigma,\nu) \,=\, iu\bar b \,-\, \frac{\sigma^2}{2}
u^2 \,+\, \int_{\R} \left(e^{iux} \,-\, 1 \,-\, \tfrac{iux}{1+x^2}\right) \,
\nu(dx).
\]
The triplet $(\bar b,\sigma,\nu)$ is called L\'evy-Khintchine characteristic or
characteristic triplet with drift-like part $\bar b\in\R$, volatility
$\sigma\geq 0$ and jump measure $\nu$, which is a non-negative $\sigma$-finite
measure on $(\R,{\mathcal B})$ with $\int \frac{x^2}{1+x^2}\, \nu(dx)<\infty$.
The function $\Psi$ is called characteristic exponent or cumulant function.

For reasons explained below, we introduce a measure $\bar\nu_\sigma$
by
\[ \bar\nu_\sigma(dx) \,=\, \sigma^2 \delta_0(dx) \,+\,
\frac{x^2}{1+x^2} \, \nu(dx),
\]
where $\delta_0$ denotes the point measure in zero. This gives
another representation of $\Psi$ in terms of $\bar b\in\R$ and the
finite Borel measure $\bar\nu_\sigma$ as
\[ \Psi(u) \,=\,
\Psi(u;\bar b,\bar\nu_\sigma) \,=\, iu\bar b \,+\, \int_\R \frac{(e^{iux} -
1)(1+x^2) - iux}{x^2} \, \bar\nu_\sigma(dx).
\]
Here we have used the continuous extension of the integrand at
$x=0$, which evaluates to $-u^2/2$. Let $P_{\bar b,\bar\nu_\sigma}$
denote the probability distribution with characteristic function
$\phi(\cdot,t;\bar b,\bar\nu_\sigma)=\exp(t \, \Psi(\cdot;\bar{b}, \bar{\nu}_\sigma))$
for some fixed $t>0$. Writing
$\mu_n\Longrightarrow \mu$ for weak convergence of the finite Borel
measures $\mu_n$ to the finite Borel measure $\mu$ on $(\R,{\mathcal B})$, the
following well-known result will be essential in the sequel
(Theorem~VII.2.9 and Remark~VII.2.10 in \citeasnoun{JS02} or
Theorem~19.1 in \citeasnoun{GK68}).

\begin{prop} \label{P2.1}
The convergence $P_{\bar{b}_n,\bar{\nu}_{\sigma,n}} \Longrightarrow P_{\bar
b,\bar\nu_\sigma}$ takes
place if and only if $\bar b_n\to \bar b$ and $\bar\nu_{\sigma,n}\Longrightarrow\bar\nu_\sigma$. \\
\end{prop}

By the scaling properties of L\'evy processes there is no loss in generality
when we suppose $t_k=k$, $k=0,\ldots,n$. We write $\phi(u;\bar
b,\bar\nu_\sigma)$ short for $\phi(u,1;\bar b,\bar\nu_\sigma)$. Let us
introduce the empirical characteristic function of the increments
\[
\widehat{\phi}_n(u):=\frac{1}{n} \sum_{t=1}^n e^{iu(X_t-X_{t-1})},\qquad
u\in\R.
\]
Since these increments are independent and identically distributed
it follows from the Glivenko-Cantelli theorem that
\begin{equation}
\label{2.1}
P_{\bar b,\bar \nu_\sigma}\left( \widehat{\phi}_n(u)
\,\mathop{\longrightarrow}\limits_{n\to\infty}\, \phi(u; \bar b,\bar
\nu_\sigma) \quad \forall u\in\R \right) \,=\, 1.
\end{equation}

We will consider minimum distance fits, that is, we intend to choose
$\widehat{\bar{b}}_n$ and $\widehat{\bar{\nu}}_{\sigma,n}$ such that, for an
appropriate metric~$d$,
\begin{equation}\label{2.2}
d(\widehat{\phi}_n, \phi(\cdot;\widehat{\bar{b}}_n,
\widehat{\bar{\nu}}_{\sigma,n})) \,=\, \inf_{\tilde{b}\in\R,\,
\tilde{\nu}_\sigma\in {\mathcal M}(\R)} d(\widehat{\phi}_n,
\phi(\cdot;\widetilde{b},\widetilde{\nu}_\sigma)).
\end{equation}
Here ${\mathcal M}(\R)$ denotes the space of all finite Borel
measures on $(\R,{\mathcal B})$. Our basic motivation for this estimation procedure
arises from the fact that an exact maximum likelihood estimator is
not feasible since there is in general no closed form expression for the
probability density of the observations available. Moreover, it is
well-known that methods based on the empirical characteristic
function can be asymptotically efficient; see \citename{FM81a}
\citeyear{FM81a,FM81b}. Since we are not sure that the infimum in
(\ref{2.2}) is always obtained, we take a sequence of positive reals
$(\delta_n)_{n\in\N}$ with $\delta_n\to 0$ as $n\to\infty$ and
choose $\widehat{\bar{b}}_n$ and $\widehat{\bar{\nu}}_{\sigma,n}$
such that
\begin{equation}
\label{2.3}
d(\widehat{\phi}_n, \phi(\cdot;\widehat{\bar{b}}_n,\widehat{\bar{\nu}}_{\sigma,n}))
\,\le\, \inf_{\tilde{b}\in\R,\, \tilde{\nu}_\sigma\in {\mathcal M}(\R)}
d(\widehat{\phi}_n, \phi(\cdot;\widetilde{b},\widetilde{\nu}_\sigma))
\,+\, \delta_n.
\end{equation}
For the metric~$d$, we assume that
\begin{equation}
\label{2.4} \lim_{n\to\infty} d(\widehat{\phi}_n, \phi(\cdot;\bar b,\bar
\nu_\sigma)) \,=\, 0 \qquad P_{\bar b,\bar\nu_\sigma}\text{-almost surely}
\end{equation}
and that the following implication holds:
\begin{eqnarray}
\label{2.5}
\left\{ \begin{array}{ll}
& \lim_{n\to\infty} d(\phi(\cdot;\bar b_n,\bar\nu_{\sigma,n}),
\phi(\cdot;\bar b,\bar\nu_\sigma)) \,=\, 0 \\
\Longrightarrow & \\
& \lim_{n\to\infty}
\int_s^t \phi(u; \bar b_n, \bar\nu_{\sigma,n}) \, du \,=\, \int_s^t \phi(u;
\bar b,\bar\nu_\sigma) \, du \quad \forall s,t\in\R.
\end{array} \right.
\end{eqnarray}

A simple example of such a distance is given by the weighted
$L^p$-norms,
\[ d(\phi_1,\phi_2) \,=\,
\Big(\int_{-\infty}^\infty |\phi_1(u)-\phi_2(u)|^p w(u) \, du \Big)^{1/p},
\]
where $p\geq 1$ and $w:\R\rightarrow (0,\infty)$ is
a continuous weight function with $\int_{-\infty}^\infty
w(u)\,du<\infty$. Then Assumption (\ref{2.4}) follows by dominated
convergence from the convergence result (\ref{2.1}), while
Assumption (\ref{2.5}) is immediate.

\section{Consistency}
\label{S3}

We derive from the triangle inequality, the definition of the
minimum-distance estimator and Assumption~(\ref{2.4}) that
\begin{eqnarray}
\label{3.0}
d(\phi(\cdot; \widehat{\bar{b}}_n, \widehat{\bar{\nu}}_{\sigma,n}),
\phi(\cdot; \bar b,\bar\nu_\sigma)) 
& \leq & d(\phi(\cdot; \widehat{\bar{b}}_n,
\widehat{\bar{\nu}}_{\sigma,n}),
\widehat{\phi}_n) \,+\, d(\widehat{\phi}_n, \phi(\cdot; \bar b,\bar\nu_\sigma)) \nonumber \\
& \leq & 2 d(\widehat{\phi}_n, \phi(\cdot; \bar b,\bar\nu_\sigma)) \,+\, \delta_n \\
& \longrightarrow & 0 \qquad P_{\bar b,\bar\nu_\sigma}\text{-a.s.} \nonumber
\end{eqnarray}
By Assumption~(\ref{2.5}) this implies for the integrated
characteristic function that
\begin{equation}
\label{2.6} P_{\bar b,\bar\nu_\sigma}\left( \int_s^t \phi(u; \widehat{\bar{b}}_n,
\widehat{\bar{\nu}}_{\sigma,n}) \, du
\,\mathop{\longrightarrow}\limits_{n\to\infty}\, \int_s^t \phi(u; \bar b,\bar
\nu_\sigma) \, du \quad \forall s,t\in\R \right) \,=\, 1.
\end{equation}
By Theorem~6.3.3 in \citeasnoun[page~163]{Chu74}, we obtain from
(\ref{2.6}) that
\begin{displaymath}
P_{\widehat{\bar{b}}_n,\widehat{\bar{\nu}}_{\sigma,n}} \,\longrightarrow_v\, P_{\bar b,\bar
\nu_\sigma} \qquad P_{\bar b,\bar \nu_\sigma}\text{-a.s.},
\end{displaymath}
where `$\longrightarrow_v$' denotes vague convergence to a possibly
defective (that is, with a mass less than~1) measure. However, since
this vague limit is a probability measure, it turns out that the
mode of convergence is actually the weak one, that is,
\begin{equation}
\label{2.7} P_{\widehat{\bar{b}}_n,\widehat{\bar{\nu}}_{\sigma,n}}
\,\Longrightarrow\, P_{\bar b,\bar \nu_\sigma} \qquad P_{\bar b,\bar
\nu_\sigma}\text{-a.s.}
\end{equation}

As an immediate consequence of Equation~(\ref{2.7}) and Proposition~\ref{P2.1}
above we obtain the following consistency result for the parameters of the
L\'evy process:

\begin{thm} \label{T3.1}
If the distance $d$ satisfies properties
\eqref{2.4} and \eqref{2.5}, then the minimum distance fit
$(\widehat{\bar{b}}_n,\widehat{\bar{\nu}}_{\sigma,n})$ is a strongly
consistent estimator, that is, with probability one we have for
$n\to\infty$
\begin{eqnarray*}
\widehat{\bar{b}}_n \rightarrow \bar b \qquad \mbox{ and } \qquad
\widehat{\bar{\nu}}_{\sigma,n} \Longrightarrow \bar\nu_{\sigma}. \\
\end{eqnarray*}
\end{thm}

\begin{rem}\label{R1}
Without further assumptions we cannot estimate the diffusion parameter $\sigma$
in a uniformly consistent way. We have for example that the stable law with
characteristic function $\phi_\alpha(u)=e^{-\abs{u}^\alpha/2}$ converges for
$\alpha\uparrow 2$ to the standard normal law ($\alpha=2$) in total variation
norm: by Scheff\'e's Lemma it suffices to show pointwise convergence of the
density functions, which follows from the $L^1$-convergence of the
characteristic functions. Hence, for $n$ observations no test can separate the
hypotheses $H_0:\alpha=2$ and $H_1:\alpha<2$. Since we have $\sigma=1$ for
$\alpha=2$ and $\sigma=0$ for $\alpha<2$, this implies for the estimation
problem uniform inconsistency in the following sense:
\begin{displaymath}
\limsup_{n\to\infty} \inf_{\hat{\sigma}_n}
\sup_{\bar b,\, \bar\nu_\sigma} P_{\bar b,\bar
\nu_\sigma}(\abs{\widehat\sigma_n-\sigma}\ge 1/2)>0,
\end{displaymath}
where the infimum is taken over all estimators based on~$n$
observations. Thus, from a statistical perspective the estimation of
the volatility $\sigma$ makes no sense, unless we restrict the class
of L\'evy processes under consideration, e.g.~to the finite
intensity case as in \cit{BR06}.
\end{rem}

The practical implementation of the minimum distance method
raises naturally the question of computational feasibility.
It is certainly not possible to compute $\widehat{\bar{\nu}}_{\sigma,n}$
by an optimisation over the full set ${\mathcal M}(\R)$.
In our simulations, for example, we approximate the measure
$\bar{\nu}_\sigma$ by measures with step-wise constant densities.
To assess the effect of such an approximation, consider a sequence
of subsets ${\mathcal M}^{(n)}\subseteq {\mathcal M}(\R)$
with the density property that there exist measures
$\widetilde{\nu}^{(n)}\in {\mathcal M}^{(n)}$ with
$\widetilde{\nu}^{(n)}\Longrightarrow \bar{\nu}_\sigma$,
as $n\to\infty$.
The definition from (\ref{2.3}) is now replaced by
\[
d(\widehat{\phi}_n, \phi(\cdot;\widehat{\bar{b}}_n,\widehat{\bar{\nu}}_{\sigma,n}))
\,\le\, \inf_{\tilde{b}\in\R,\, \tilde{\nu}_\sigma\in {\mathcal M}^{(n)}}
d(\widehat{\phi}_n, \phi(\cdot;\widetilde{b},\widetilde{\nu}_\sigma))
\,+\, \delta_n.
\]
We obtain instead of (\ref{3.0}) that
\begin{eqnarray*}
d(\phi(\cdot; \widehat{\bar{b}}_n, \widehat{\bar{\nu}}_{\sigma,n}),
\phi(\cdot; \bar b,\bar\nu_\sigma)) 
& \leq & 2 d(\widehat{\phi}_n, \phi(\cdot; \bar b,\bar\nu_\sigma)) \,+\, \delta_n 
\,+\, d( \phi(\cdot; \bar b,\bar\nu_\sigma), \phi(\cdot; \bar b,\widetilde{\nu}^{(n)}) )
\\
& \longrightarrow & 0 \qquad P_{\bar b,\bar\nu_\sigma}\text{-a.s.}
\end{eqnarray*}
Hence, we obtain in complete analogy to Theorem~\ref{T3.1} that with
probability one for $n\to\infty$
\begin{eqnarray*}
\widehat{\bar{b}}_n \rightarrow \bar b \qquad \mbox{ and } \qquad
\widehat{\bar{\nu}}_{\sigma,n} \Longrightarrow \bar\nu_{\sigma}. \\
\end{eqnarray*}

Given the existence of certain moments for $P_{\bar b,\bar\nu_\sigma}$, we could also search our minimum-distance estimator in the class of those parameter values that fit the empirical moments. Using a similar error decomposition
and the consistency of the empirical moments, this approach will also yield consistent estimators under mild conditions on the distance $d$.

\section{A rate-optimal estimator}
\label{S4}

\subsection{The construction}
\label{S4.1}

In this section we intend to devise estimators which attain optimal rates of
convergence. We henceforth restrict the class of L\'evy processes to those with
finite second moments. This is equivalent to requiring that the L\'evy measure
satisfies $\int x^2\nu(dx)<\infty$. In this case the following
reparametrisation of the characteristic exponent is much more convenient:
\[
\Psi(u; b,\sigma,\nu) \,=\, iu b \,-\, \frac{\sigma^2}{2} u^2 \,+\, \int_{\R}
(e^{iux} \,-\, 1 \,-\, iux) \, \nu(dx),
\]
where the parameter $b=\bar b+\int_{\R} (x-\frac{x}{1+x^2}) \nu(dx)$ denotes
now indeed the mean trend because of $\E[X_1]=-i\phi'(0)=b$. Let us mention that this is the original Kolmogorov canonical representation of a L\'evy process \cite{Kol32}, the historial background of which is nicely
exposed by \cit{MR06}. Instead of
$\bar\nu_\sigma$, we consider the finite measure $\nu_\sigma$ defined by
\[
\nu_\sigma(dx) \,=\, \sigma^2 \delta_0(dx) \,+\, x^2 \, \nu(dx),
\]
which allows the nice identity
$\Var(X_1)=-\phi''(0)+\phi'(0)^2=\nu_\sigma(\R)$. From now on, we shall express
the characteristic exponent in terms of $(b,\nu_\sigma)$:
\[
\Psi(u) \,=\, \Psi(u;b,\nu_\sigma) \,=\, iub \,+\, \int_\R \frac{e^{iux} - 1 -
iux}{x^2} \, \nu_\sigma(dx).
\]

While~$b$ can be easily estimated by $\frac{1}{n} \sum_{t=1}^n (X_t - X_{t-1})
=X_n/n$, the construction of an optimal nonparametric estimator of~$\nu_\sigma$
requires more work. Before we start with our search for optimal rates of
convergence for estimators of~$\nu_\sigma$, we have to decide about an
appropriate metric to measure the deviation of any potential
estimator~$\widetilde{\nu}_{\sigma,n}$ from its target~$\nu_\sigma$.

The parameter $\nu_\sigma$ lies in the space of finite Borel
measures, which is naturally equipped with the total variation norm.
As we have seen above in the consistent estimation problem for
$\sigma$, this topology is too strong here. Moreover, we are usually
not interested in the problem of estimating $\nu_\sigma$ itself, but
rather in estimating integrals $\int f\,d\nu_\sigma$ for certain
integrands $f$. In mathematical finance for example, the so-called
$\Delta$ in the quadratic hedging approach requires calculating
$\int \frac{C(t,S(1+z))-C(t,S)}{Sz}\,\nu_\sigma(dz)$, where $C(t,S)$
denotes the option price at time $t$ and $S$ the corresponding stock
price, cf. Proposition~10.5 in \cit{CT04}. This is why we choose to
measure the performance of our estimator by metrizing weak
convergence with certain classes $F$ of continuous test functions
$f$:
\[
l(\widetilde{\nu}_{\sigma,n}, \nu_\sigma) \,=\, \sup\left\{ \left|
\int f\, d\widetilde{\nu}_{\sigma,n} \,-\,  \int f\, d\nu_\sigma
\right|: \; f\in F \right\}.
\]
Note that for any class $F$ of uniformly bounded, equicontinuous
functions consistency with respect to weak convergence implies
$l(\widetilde{\nu}_{\sigma,n}, \nu_\sigma)\to 0$ \cite[Cor.
11.3.4]{Dud89}.
For instance, the bounded Lipschitz metric is generated by
the test functions of Lipschitz norm less than one. 

Let us introduce the Fourier transform for functions $f\in L^1(\R)$
or measures $\mu\in{\mathcal M}(\R)$ by
\[
\FT f(u)=\int f(x)e^{iux}dx,\qquad \FT
\mu(u)=\int e^{iux}\mu(dx),\quad u\in\R.
\]
Note that we have by Parseval's equality
\[
\int f\, d\nu_\sigma \,=\, \frac{1}{2\pi}\int_{-\infty}^\infty \FT
f(u) \overline{\FT \nu_\sigma(u)} \, du,
\]
provided $\FT f\in L^1(\R)$ \cite[Theorem~VI.2.2]{Ka76}.
Estimation of $\nu_\sigma$ turns out to be particularly transparent when we
employ the fact that
\[
\Psi''(u)=\frac{d^2}{du^2}\int\frac{e^{iux}-1-iux}{x^2}\,\nu_\sigma(dx)=-\FT\nu_\sigma(u),
\]
and consequently
\begin{equation}\label{4.1} \FT \nu_\sigma(u) \,=\, -
\frac{d^2}{du^2} \log( \phi(u)) \,=\, \frac{ {\phi'(u)}^2 }{ \phi(u)^2 } \,-\,
\frac{ \phi''(u) }{ \phi(u) }.
\end{equation}
Recall that $\int x^2\nu(dx)<\infty$ implies $\E[X_t^2]<\infty$ and
hence $\phi\in C^2$. Moreover, in order to recover $\Psi$ from $\varphi$
we use the distinguished logarithm of
the complex-valued function $u\mapsto \phi(u)$, which is required to
ensure $\log(\phi(0))=0$ and continuity of $u\mapsto \log(\phi(u))$,
cf.~\cit{CT04}. This formula indicates that estimating $\nu_\sigma$
is strongly related to estimating $\phi$ in a $C^2$-sense. Before we
study rates of convergence, we need to investigate uniform rates of
convergence of the empirical characteristic
function~$\widehat{\phi}_n$ and its derivatives.

\subsection{Estimating the characteristic function}\label{S41b}

For i.i.d.~random variables $(Z_t)_{t\in\N}$, denote by
\[
C_n(u) \,:=\, n^{-1/2} \sum_{t=1}^n \left( e^{iuZ_t} \,-\, \E[e^{iuZ_1}]
\right)
\]
the normalized characteristic function process.
Furthermore, denote by $C_n^{(k)}$ its $k$th derivative which exists
if $\E|Z_1|^k<\infty$. For an appropriate weight function
$w:\R\longrightarrow[0,\infty)$, we consider
\[
\E \| C_n^{(k)} \|_{L_\infty(w)} \,:=\, \E \sup_{u\in\R}
\left\{ |C_n^{(k)}(u)| w(u) \right\}.
\]
For every $k\ge 0$ we have the following general result.

{\thm \label{T4.1}
Suppose that $(Z_t)_{t\in\N}$ are i.i.d.~random
variables with $\E|Z_1|^{2k+\gamma}<\infty$ for some $\gamma>0$
and let the weight function be
defined as $w(u)=(\log(e+|u|))^{-1/2-\delta}$ for some
$\delta>0$. Then
\begin{eqnarray*}
\sup_{n\ge 1} \E \| C_n^{(k)} \|_{L_\infty(w)} \,<\, \infty.\\
\end{eqnarray*}
}

Its proof is given in Section~\ref{SPT4.1}. Let us mention that the logarithmic
decay of the weight function $w$ is in accordance with the well known result
that $\widehat\phi_n\to\phi$ a.s. holds uniformly on intervals $[-T_n,T_n]$
whenever $\log(T_n)/n\to 0$, cf. \cit{Cs83}.

\subsection{Upper risk bounds}

In view of (\ref{4.1}) and Theorem~\ref{T4.1}, we define our estimators of
$b$ and $\nu_\sigma$ by a minimum distance fit based on a
weighted $C^2$-norm. Defining
\[ d^{(2)}(\phi_1, \phi_2)
\,:=\, \sum_{k=0}^2 \| \phi_1^{(k)} \,-\, \phi_2^{(k)} \|_{L^\infty(w)},
\]
we choose the estimators
 $\widehat{b}_n\in\R$ and $\widehat{\nu}_{\sigma,n}\in {\mathcal M}(\R)$
such that
\begin{equation}
\label{4.2}
d^{(2)}\left( \phi(\cdot; \widehat{b}_n,\widehat{\nu}_{\sigma,n}),
\widehat{\phi}_n \right)
\,\le\,
\inf_{\tilde{b}\in\R,\, \tilde{\nu}_\sigma\in {\mathcal M}(\R)}
d^{(2)}\left( \phi(\cdot;
\widetilde{b},\widetilde{\nu}_\sigma), \widehat{\phi}_n \right) \,+\, \delta_n,
\end{equation}
where $\delta_n\to 0$ as $n\to\infty$. We verify by Theorem~\ref{T4.1}
that $d^{(2)}$ satisfies Assumptions (\ref{2.4}) and
(\ref{2.5}), hence, Theorem \ref{T3.1} gives immediately a consistency
result. Moreover, with the choice $\delta_n=O(n^{-1/2})$ these estimators will
turn out to be rate-optimal.

While $b$ can always be estimated at rate $n^{-1/2}$, rates of
convergence of $\int f\, d\widehat{\nu}_{\sigma,n}$ as an estimator
of $\int f\,d\nu_\sigma$ depend both on the smoothness of~$f$ and on
the decay of $|\phi(u)|$ as $|u|\to\infty$. For the function~$f$, we
will assume that it belongs to the class
\[
F_{s} \,:=\, \left\{f: \;\; \int (1+|u|)^s |\FT f(u)|\, du \,\leq\, 1 \right\},
\]
for some $s\ge 0$. Note that $\int |\FT f(u)|\,du\le 1$  implies by the
Riemann-Lebesgue Lemma that~$f$ is continuous with $\|f\|_\infty\le 1$. By
Fourier theory the condition $f\in F_{s}$ is slightly stronger than requiring
$f\in C^s$ with $\norm{f}_{C^s}\le 1$ for a suitable norming of $C^s$. We
therefore introduce a loss function for an estimator $\widehat{\mu}$ of the finite
measure $\mu$ by
\[ \ell_s(\widehat\mu,\mu):=\sup_{f\in F_s} \left| \int f\,
d\widehat{\mu} \,-\, \int f\, d\mu \right|.
\]
Note that by duality the loss $\ell_s$ can be interpreted as a negative
smoothness norm of order~$-s$.

The faster $|\phi(u)|$ decays, the more difficult it will be to estimate
$\nu_\sigma$. We consider in particular the following three cases:
\begin{itemize}
\item[{\bf (a)}] {\bf Gaussian part}\\
If $\sigma^2>0$, then the characteristic function $\phi$ has Gaussian tails,
i.e. 
\[ 
\log | \phi(u) | \,=\, \Re\left( \log \varphi(u) \right)
\,=\, - \sigma^2 u^2/2 (1 \,+\, o(u)) , \qquad \mbox{as } |u|\to\infty.
\]
(To see this, note that $F(x,u):=(e^{iux}-1-iux)/(ux)^2$ is uniformly bounded with
$\lim_{|u|\to\infty}F(x,u)=0$
for $x\neq 0$ such that by dominated convergence 
$\lim_{|u|\to\infty} \int_{\R} F(x,u)x^2\, \nu(dx)=0$
and thus $\log \varphi(u)=-\sigma^2u^2/2+o(u^2)$.)

\item[{\bf (b)}] {\bf Exponential decay}\\
Here the characteristic function $\phi$ decays at most exponentially, i.e. for
some $\alpha>0$, $C>0$,
\[ | \phi(u) | \,\ge\, C e^{-\alpha |u|}, \qquad \mbox{for all } u\in\R.
\]
Examples of distributions with this property include normal inverse
Gaussian \cite[page~117]{CT04}, and generalized tempered stable
distributions \cite[page~122]{CT04}.

\item[{\bf (c)}] {\bf Polynomial decay}\\
In this case the characteristic function satisfies, for some
$\beta\ge 0$, $C>0$,
\[ | \phi(u) | \,\ge\, C(1 \,+\, |u|)^{-\beta}, \qquad \mbox{for all } u\in\R.
\]
Typical examples for this are the compound Poisson distribution, the
gamma distribution, the variance gamma distribution and the
generalized hyperbolic distribution \cite[pages~75, 116, 117,
127]{CT04}.
\end{itemize}

The proof of the following main theorem is postponed to Section~\ref{SPT4.2}.

\begin{thm} \label{T4.2}
Suppose that $\E_{b,\nu_\sigma}|X_1|^{4+\gamma}<\infty$ for some $\gamma>0$. We choose the
weight function~$w$ as $w(u)=(\log(e+|u|))^{-1/2-\delta}$, where $\delta$ is
any positive number. The estimators $\widehat{b}_n$ and
$\widehat{\nu}_{\sigma,n}$ of $b$ and $\nu_\sigma$, respectively, are chosen
according to~(\ref{4.2}) with $\delta_n=O(n^{-1/2})$. Then
\[\E_{b,\nu_\sigma}
|\widehat{b}_n \,-\, b| \,=\, O( n^{-1/2} )
\]
and for any $s>0$
\[
\ell_s(\widehat{\nu}_{\sigma,n},\nu_\sigma) \,=\, O_{P_{b,\nu_\sigma}}\left(
n^{-1/2} \cdot \sup_{u\in [0,U_n]} \left\{
\frac{(1+u)^{2-s}}{w(u)|\phi(u;b,\nu_\sigma)|} \right\} \right),
\]
where
\[U_n:=\inf\left\{u>0: \; \frac{(1+u)^2 n^{-1/2}}{w(u)|\phi(u;b,\nu_\sigma)|}\ge 1 \right\}.
\]
The constants in the risk bounds depend continuously on $\abs{b}$ and
$\nu_\sigma(\R)$. In the specific cases we obtain the following rates of
convergence for $\ell_s(\widehat{\nu}_{\sigma,n},\nu_\sigma)$ in
$P_{b,\nu_\sigma}$-probability:
\begin{description}
\item[{\bf (a) Gaussian part}] $(\log n)^{-s/2}$
\item[{\bf (b) Exponential decay}] $(\log n)^{-s}$
\item[{\bf (c) Polynomial decay of order $\beta\ge 0$}]
$ [(\log n)^{1/2+2\delta} n^{-1/2}]^{s/\beta}\vee n^{-1/2}$.\\
\end{description}
\end{thm}

\begin{rem}
The results are presented for convergence in probability, but the proof
immediately yields convergence of moments of order $1/2$ of the loss in cases
(a), (b), cf. Equation \eqref{pt42.2}. Higher moments are achieved whenever the
order of the moment bound in Theorem \ref{T4.1} can be increased.
\end{rem}

\subsection{Lower risk bounds} \label{S4.2}

We prove that the rates of convergence obtained in Theorem \ref{T4.2} for cases
(a), (b), (c) are optimal, at least up to a logarithmic factor in the latter
case. The proof in Section \ref{SPT4.4} can be naturally generalized to cover
further decay scenarios of the characteristic function.

\begin{thm} \label{T4.4}
For $C,\bar{C}>0$ large enough and for any $\alpha>0$, $\beta\ge 0$
introduce the following nonparametric classes of $\nu_\sigma$:
\begin{align*}
{\mathcal A}(C,\sigma)&:=\left\{\nu_\sigma\in{\mathcal
M}(\R)\,\Big|\,\nu_\sigma(\R)\le
C\right\} \text{ ($\sigma>0$),}\\
{\mathcal B}(C,\alpha)&:=\left\{\nu_\sigma\in{\mathcal M}(\R)
\,\Big|\,\nu_\sigma(\R)\le
C,\,|\phi(u)|\ge \bar{C} e^{-\alpha \abs{u}}\right\}\text{ ($\sigma=0$),}\\
{\mathcal C}(C,\bar{C},\beta)&:=\left\{\nu_\sigma\in{\mathcal M}(\R)
\,\Big|\,\nu_\sigma(\R)\le C,\,\abs{\phi(u)}\ge \bar{C}^{-1}
(1+\abs{u})^{-\beta}\right\}\text{ ($\sigma=0$)}.
\end{align*}
Then we obtain for some fixed $b\in\R$ and for any $s>0$ the following minimax
lower bounds, where $\widetilde\nu_{\sigma,n}$ denotes any estimator of
$\nu_\sigma$ based on $n$ observations:
\begin{align*}
\text{{\bf (a)} } & \exists\eps>0:&&
\liminf_{n\to\infty}\inf_{\tilde\nu_{\sigma,n}}\sup_{\nu_\sigma\in
{\mathcal A}(C,\sigma)} P_{b,\nu_\sigma} \Big((\log n)^{s/2}
\ell_s(\widetilde{\nu}_{\sigma,n},\nu_\sigma)>\eps\Big)>0,\\
\text{{\bf (b)} } & \exists\eps>0: &&
\liminf_{n\to\infty}\inf_{\tilde\nu_{\sigma,n}}\sup_{\nu_\sigma\in
{\mathcal B}(C,\alpha)} P_{b,\nu_\sigma}
\Big( (\log n)^{s}\ell_s(\widetilde{\nu}_{\sigma,n},\nu_\sigma)>\eps\Big)>0,\\
\text{{\bf (c)} }& \exists\eps>0:
&&\liminf_{n\to\infty}\inf_{\tilde\nu_{\sigma,n}}
\sup_{\nu_\sigma\in {\mathcal C}(C,\bar{C},\beta)}
P_{b,\nu_\sigma}\Big( n^{(s/2\beta)\wedge
(1/2)}\ell_s(\widetilde{\nu}_{\sigma,n},\nu_\sigma)
>\eps\Big)>0.\\
\end{align*}
\end{thm}

\subsection{Discussion}\label{S4.3}

The convergence rates for $\widehat{\nu}_{\sigma,n}$ can be understood in
analogy with a deconvolution problem where the Fourier transform of
the error density decays like the characteristic function $\phi$ in
our case, see e.g.~\cit{Fan91}. The interesting point here is that
this decay property is not assumed to be known and depends on the
parameters to be estimated. At first sight, it is rather surprising
that our minimum distance estimator adapts automatically to the
decay of $\phi$, even for the whole range of loss functions
$\ell_s$, $s>0$. This is due to the fact that the noise level in the
empirical characteristic function $\widehat\phi_n$ is of the same
size for different frequencies and this is where we fit our
estimator. In contrast, when fitting the characteristic exponent
$\Psi$, which is more attractive from a computational point of view
and for example advocated in \cit{Hol05}, we face a highly
heteroskedastic noise level in $\log(\widehat{\phi}_n(u))$ governed
by $\abs{\phi(u)}^{-1}$ because of
$\log(\widehat{\phi}_n(u))-\Psi(u)\approx
(\widehat{\phi}_n(u)-\phi(u))/\phi(u)$.

Another point of view on our estimation problem is that we want to estimate the
linear functional $\int f\,d\nu_\sigma$ based on an inverse problem setting for
estimating $\nu_\sigma$. In an abstract Hilbert scale context, adaptive
estimation  for this has been considered by \cit{GP03} and their rate for the
polynomially ill-posed case reads in our notation
$(n/\log(n))^{-(r+s)/(2r+2\beta)}\vee n^{-1/2}$, with $r$ the regularity of
$\nu_\sigma$, $s$ the regularity of $f$ and $\beta$ the degree of
ill-posedness. In our case, we measure the regularity $s$ of $f$ in the Fourier
domain by an $L^1$-criterion such that a dual $L^\infty$-criterion for the
regularity of $\nu_\sigma$ yields $r=0$ because $\norm{\FT \nu_\sigma}_\infty$
is finite. Hence, the rate $(n/\log(n))^{-s/2\beta}\vee n^{-1/2}$, up to the
logarithmic factor of power $\delta$, obtained in case (c) of Theorem
\ref{T4.2}, confirms this analogy. We suspect that the gap by a logarithmic
factor in the polynomial case between our upper and lower bound is mainly due
to a suboptimal lower bound, because $\ell_s$ can be expressed in the Fourier
domain via
\[ \ell_s(\widehat\mu,\mu)=\sup_{f\in F_s}\babs{\int \FT f(u)
\overline{\FT(\widehat\mu-\mu)(u)}\,du}=\sup_{u\in\R}(1+\abs{u})^{-s}\abs{\FT(\widehat\mu-\mu)(u)},
\]
giving a supremum-type norm.

It is certainly remarkable that no regularisation parameter is
involved in our estimation procedure which becomes more intuitive by
noticing that the results of Section~\ref{S3} imply consistency
already for $s=0$. On the other hand, better rates of convergence
can be obtained when we restrict the model to measures $\nu_\sigma$
which have a regular Lebesgue density~$g_\sigma$. A natural plug-in
approach yields the kernel-type estimator
$\widehat{g}_{\sigma,n,h}(x):=K_h\ast\widehat{\nu}_{\sigma,n}(x)$,
convolving the minimum-distance estimator with a smooth kernel $K_h$
of bandwidth $h>0$. Noting that $\int f
\widehat{g}_{\sigma,n,h}=\int (f\ast
K_h)\,d\widehat{\nu}_{\sigma,n}$, we infer that the bound on the
stochastic error
\[\babs{\int f (\widehat{g}_{\sigma,n,h}-K_h\ast
g_\sigma)}=\babs{\int (f\ast K_h)
d(\widehat\nu_{\sigma,n}-\nu_\sigma)}
\]
is controlled by the regularity of $f\ast K_h$. To be
more specific, consider a function $f$ with $\abs{\FT f(u)}\asymp
(1+\abs{u})^{-s-1}$ (e.g. $f(x)=e^{-\abs{x}}$ with $s=1$), suppose
$\sup_u(1+\abs{u})^r\abs{\FT g_\sigma(u)}<\infty$ for $r>0$ and
assume polynomial decay of order $\beta\ge s$ of the characteristic
function. Then $\int (1+\abs{u})^\beta \abs{\FT f(u)\FT
K_h(u)}\,du\asymp h^{s-\beta}$ holds such that $ch^{-s+\beta} f\ast
K_h$ lies in $F_{\beta}$, $c>0$ some small constant, and Theorem
\ref{T4.2} implies that
\[
\babs{\int f(\widehat{g}_{\sigma,n,h}-K_h\ast g_\sigma)}  \,=\, O_P\left(
h^{s-\beta}n^{-1/2} (\log n)^{1/2+2\delta}\right).
\]
Together with an easy bias estimate of order $h^{s+r}$ this yields for the
estimation error $\abs{\int f(\widehat{g}_{\sigma,n,h}- g_\sigma)}$ up to
logarithmic factors the rate $n^{-(r+s)/(2r+2\beta)}$, provided the bandwidth
is chosen in an optimal way. We conclude that our results also allow to obtain
risk bounds under smoothness restrictions, which are coherent with the abstract
results in \cit{GP03}. The rates should also be compared with the case of
continuous-time observations on $[0,T]$, where \cit{FH06} obtained the
classical nonparametric rate $T^{-r/(2r+1)}$ for estimating $g_\sigma$ on a
bounded interval.

\section{Implementation}\label{S4bis}

Although the main focus of our work is theoretical, we point out how
the minimum distance estimator can be implemented and show a
numerical example. The main computational problem is that the
procedure requires to minimize a nonlinear functional over the space
of all finite measures. One possibility is to use a global
optimisation procedure, e.g. based on simulated annealing, cf.
\cit{HY03} for an application to minimum-distance fits based on
characteristic functions. Here we shall look for a good preliminary
estimator and minimize the $d^{(2)}$-criterion locally around this
pilot estimator, which turns out to be more stable in simulations than global
optimisation routines.

We use the identification formula \eqref{4.1} to build a first-stage
plug-in estimator $(\widetilde b_n,\widetilde\nu_{\sigma,n})$. While
the mean $b$ will be easily estimated by
\begin{displaymath}
\widetilde{b}_n \,:=\, \frac{1}{n} \sum_{t=1}^n (X_t-X_{t-1})
\,=\, X_n/n,
\end{displaymath}
we have to be more careful with an estimator of $\nu_\sigma$. Since
$\FT\nu_\sigma(u)=\phi''(u)/\phi(u)-(\phi'(u)/\phi(u))^2$ one might
be tempted to estimate its Fourier transform just by plugging in the
empirical characteristic function $\widehat{\phi}_n$ for~$\phi$. It
turns out, however, that the occurrence of $\widehat{\phi}_n(u)$ in
the denominator might have unfavorable effects, particularly if
$|\phi(u)|$ is small. To get some intuition for a possible remedy,
consider the problem of estimating $1/\phi(u)$.
$1/\widehat{\phi}_n(u)$ is certainly a good estimator as long as
$|\widehat{\phi}_n(u)|$ is not too small. On the other hand, since
the noise level of $\widehat{\phi}_n(u)$ is $O(n^{-1/2})$ we should
no longer rely on $1/\widehat{\phi}_n(u)$ if
$\widehat{\phi}_n(u)=O(n^{-1/2})$. To take this into account, one
can use $\1_{\{\abs{\hat\phi_n(u)}\ge\kappa
n^{-1/2}\}}/\widehat{\phi}_n(u)$ as an estimator for $1/\phi(u)$
which can be proven to satisfy
\begin{displaymath}
\E_{b,\nu_\sigma} \left|
\frac{ \1_{\{\abs{\hat{\phi}_n(u)}\ge\kappa n^{-1/2}\}} }{ \widehat{\phi}_n(u) }
\,-\, \frac{1}{\phi(u)} \right|^p 
\,=\, O\left( \left( \frac{n^{-1/2}}{|\phi(u)|^2} \wedge \frac{1}{|\phi(u)|}
\right)^p \right),
\end{displaymath}
for any positive threshold value~$\kappa$ and all $p\in\N$.
This is what we can at best expect from an estimator of $1/\phi(u)$.
Using this idea we define our preliminary estimator of $\FT\nu_\sigma(u)$ by
\begin{equation}\label{5.1}
\FT\widetilde{\nu}_{\sigma,n}(u)
\,:=\, \left( \frac{\widehat{\phi}_n''(u)}{\widehat{\phi}_n(u)}
\,-\, \left( \frac{\widehat{\phi}_n'(u)}{\widehat{\phi}_n(u)} \right)^2 \right)
\1_{\{\abs{\hat{\phi}_n(u)}\ge\kappa n^{-1/2}\}},
\end{equation}
where $\kappa$ is a positive constant. In Section~\ref{SPP5.1} below
we shall prove the following result.

\begin{prop}\label{P5.1}
We have $\E_{b,\nu_\sigma}(\tilde b_n-b)^2=O(n^{-1})$ and for $u\in\R$
\begin{eqnarray*}
\E_{b,\nu_\sigma} \left| \FT\widetilde{\nu}_{\sigma,n}(u) \,-\,
\FT\nu_\sigma(u) \right| 
\,=\, O\left( \left( \frac{n^{-1/2}}{|\phi(u)|} \wedge 1 \right)
\left( 1 \,+\, |\Psi'(u)|^2 \right) \right). \\
\end{eqnarray*}
\end{prop}

This will give pointwise rates of convergence in a similar fashion
as before and serves well as a starting point of a local
optimisation routine. Note that this pilot estimator is very easy
and fast to implement. Yet, it has certain drawbacks, most
importantly $\FT\widetilde{\nu}_{\sigma,n}$ is usually not
positive semidefinite so that $\widetilde{\nu}_{\sigma,n}$ is not 
necessarily a non-negative measure.

\begin{figure}[t]
\centering
\includegraphics[width=7.3cm]{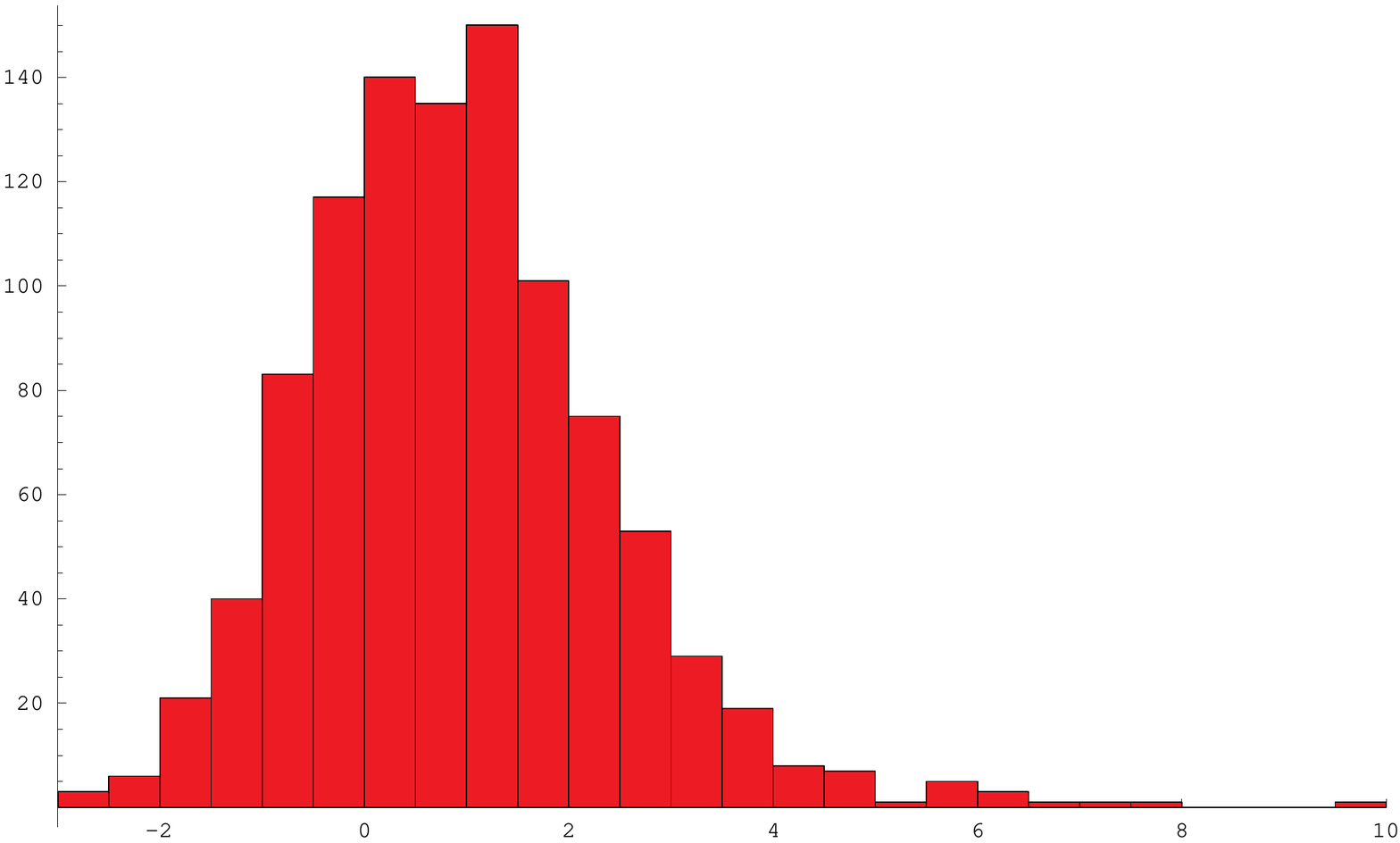}
\includegraphics[width=7.3cm]{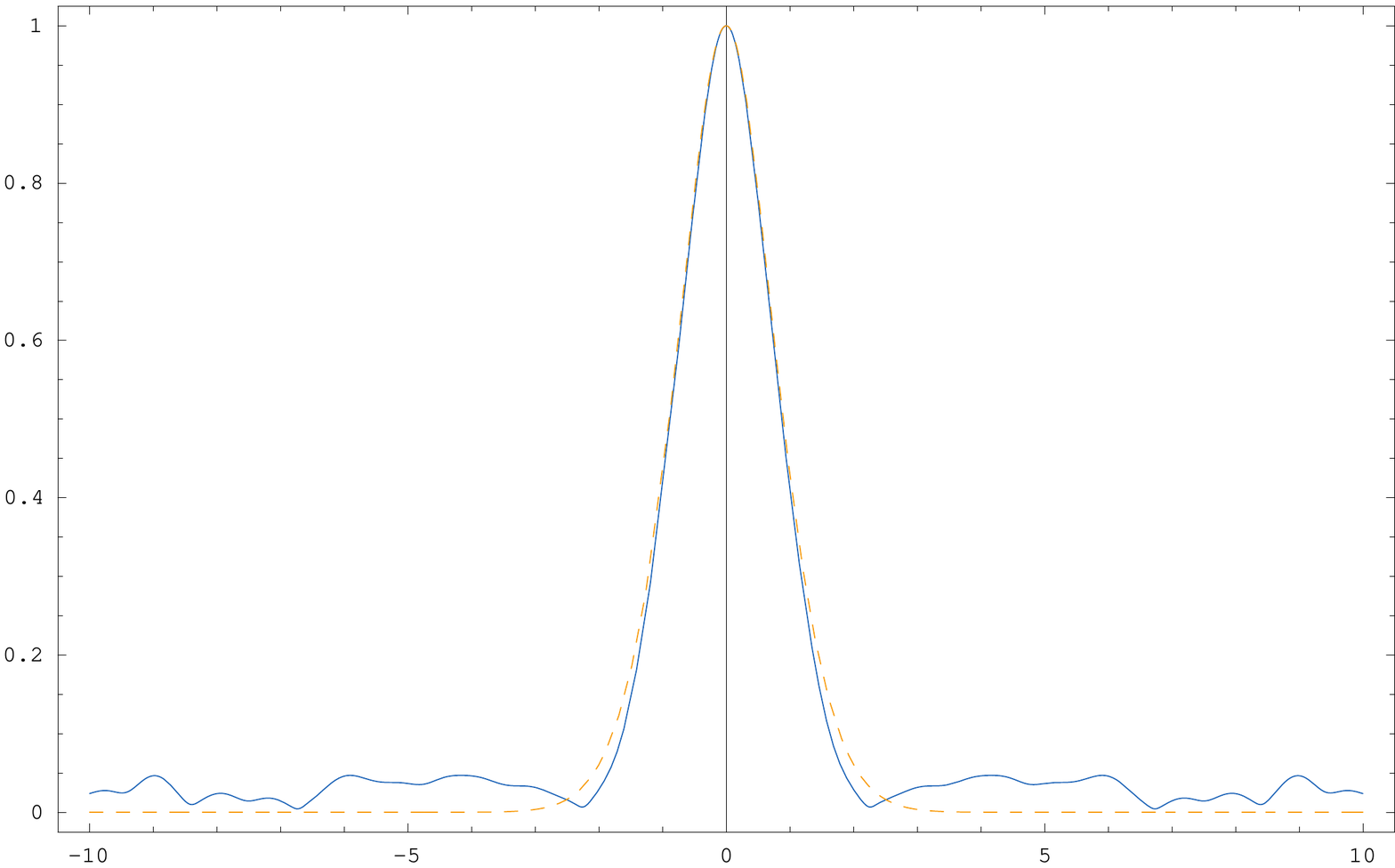}\\[-2cm]
\caption{Left: Histogram of the data. Right: modulus of the
empirical (solid blue) and true (dashed orange) characteristic
function. }\label{Figure1}
\end{figure}

In practice, our two-stage procedure works reasonably well. For a
numerical example we simulate a L\'evy process $(X_t)_{t\geq 0}$ with
$\sigma=1$, $b=1$ and $\nu(dx)=x^{-1}e^{-x}\,\1_{\{x>0\}}dx$. The process $X$ is a superposition of 
an infinite-intensity Gamma process and a standard Brownian motion. The law of
its increments $X_t-X_{t-1}$ is the convolution of an $N(0,1)$- and
an $\Exp(1)$-distribution. We have $n=1000$ observations, see Figure
\ref{Figure1}(left) for a histogram of the increments. The sample is
rather disperse with some increments close to $10$ and a sample mean
of $\tilde b_n=0.936$ (true $b=1$). The true characteristic function
has Gaussian decay and its absolute value is shown together with
that of the empirical characteristic function in Figure
\ref{Figure1}(right).

We discretize the pilot estimate $\widetilde{\nu}_{\sigma,n}$ of the jump measure
by using a Haar wavelet basis on the interval $[-10,10]$ with 15
basis functions. Moreover, we allow for a point measure in zero to
have a better resolution there. Its pilot mass is set to zero. Using
the {\tt FindMinimum} local optimisation procedure in Mathematica,
we minimize the $d^{(2)}$-criterion locally around the discretized
pilot estimator, constraining to non-negative L\'evy measures. In
Figure \ref{Figure2}(left) we display for the given data the
imaginary part of the empirical characteristic function together
with the imaginary parts of
the other characteristic functions of interest (true, pilot,
final estimator). The errors in fitting the real part are less pronounced
because there a less oscillations around zero (note $(\Re \phi)'(0)=0$). 
Typically, the pilot estimator gives already a
reasonably good fit and the final estimator has a characteristic
function which is closer to the empirical characteristic function
than the true one.

\begin{figure}[t]
\centering
\includegraphics[width=7.3cm]{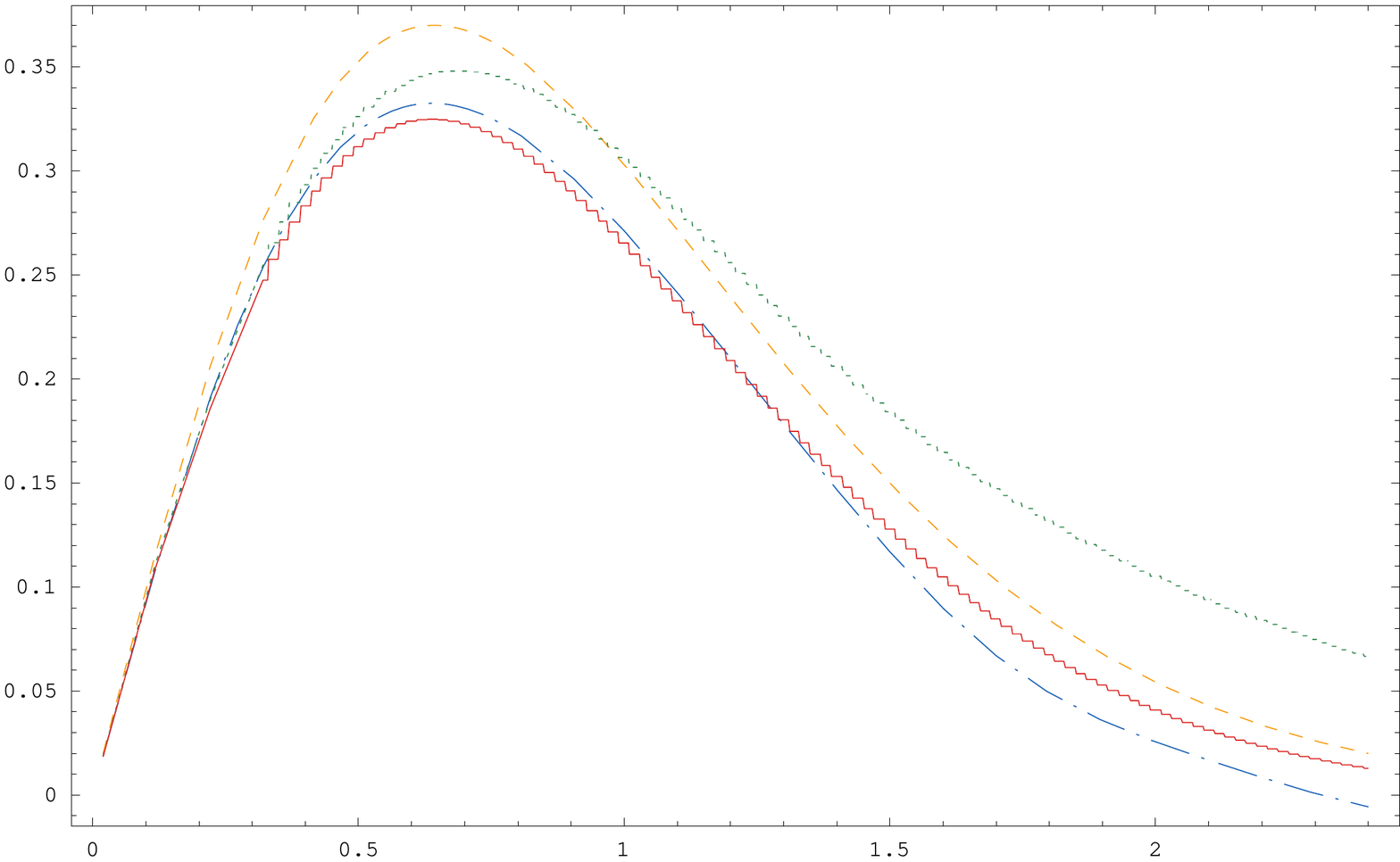}
\includegraphics[width=7.3cm]{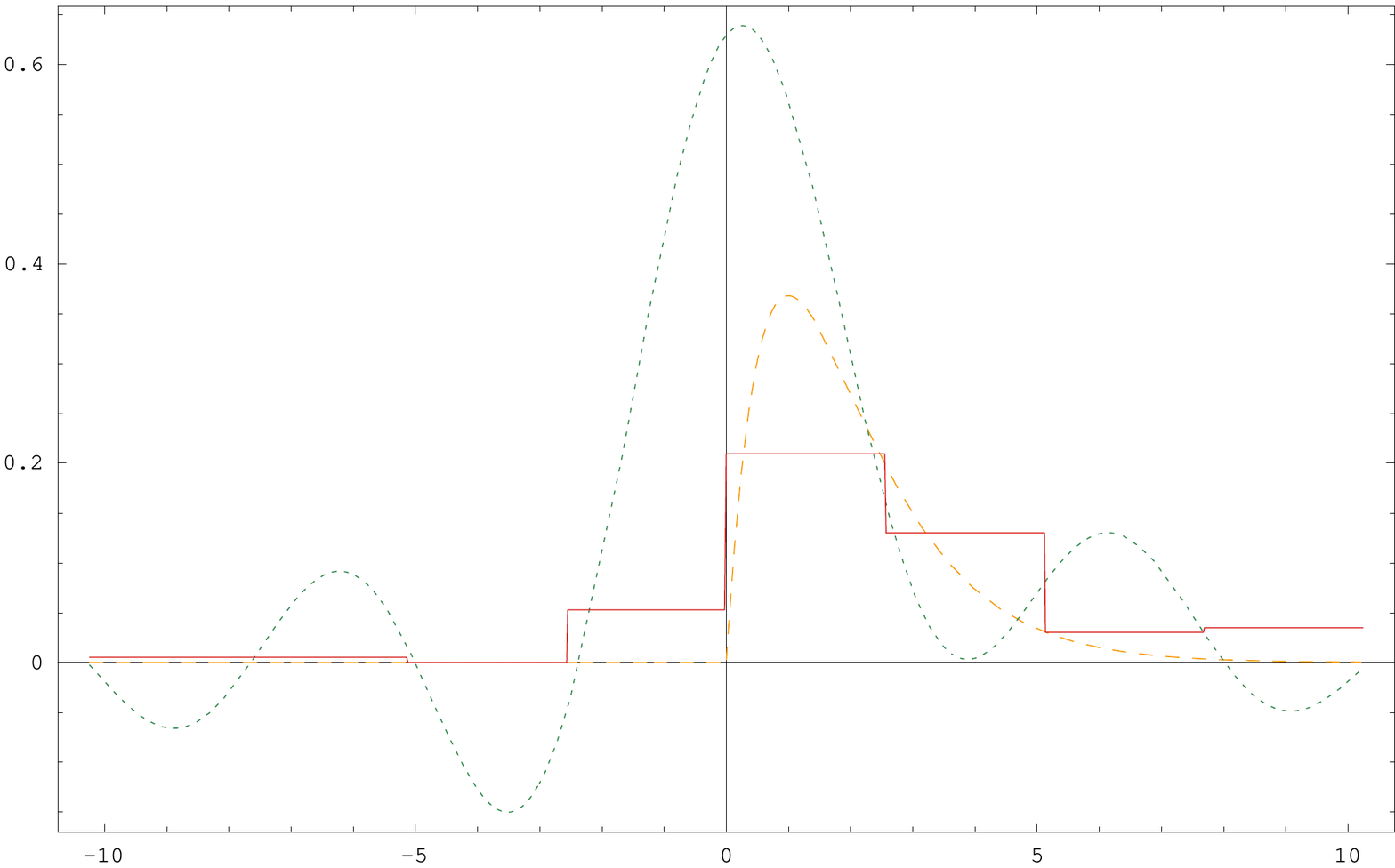}\\[-2cm]
\caption{Left: Imaginary part of the empirical (dot-dashed blue),
true (orange dashed), pilot (dotted green) and final estimated (red
solid) characteristic function. Right: pilot (dotted green), final
(red solid) estimator and true (dashed orange) density of $\nu_\sigma$; the
pilot estimator does not have a point mass in zero. }\label{Figure2}
\end{figure}

Figure \ref{Figure2}(right) finally shows the densities of the
rescaled L\'evy measures $\nu_\sigma$, but suppresses the point masses in zero. Note that
the original L\'evy density and also its plug-in estimators have a singularity at zero
because of $\nu(dx)=x^{-2}\nu_\sigma(dx)$ for $x\not=0$. The
parameters are estimated as $\widehat{b}_n=0.922$ (true $b=1$) and
$\widehat\nu_{\sigma,n}(\{0\})^{1/2}=1.092$ (true $\sigma=1$). The
pilot estimator has no point mass in zero and its density is
therefore large around zero. It is seen that the final estimator
improves upon the pilot estimator, in particular by excluding
negative values and catching the point mass in zero. Given 1000
observations and a Gaussian deconvolution problem, the estimation
problem is quite hard. The rough, step-wise form of the final
estimator is not so pleasant for the human eye, but we only want to
use this estimator as an integrator of smooth functions and, as
discussed above, we could apply a kernel to obtain a smooth density
function. As an example for a functional to be estimated, we calculated 
$\int_1^\infty x^{-2}\widehat{\nu}_\sigma(dx)$ which estimates $\nu([1,\infty))$, 
the probability of jumps larger than one. In this sample, 
the true value 0.22 was estimated by 0.16. Let us remark that the high-frequency estimator,
using the relative frequency of increments $X_{t}-X_{t-1}$ that are larger than one,
yields the estimate 0.46. The large error of the latter confirms a strong violation of the underlying 
high-frequency assumption that between two observations very rarely more than one larger jump occurs
and that the diffusion part is negligible. Hence, the frequency of the observations must indeed 
be considered as low for the construction of the estimator.

\section{Proofs}
\label{S5}

\subsection{Proof of Theorem~\ref{T4.1}}\label{SPT4.1}

We begin the proof with a few definitions. Given two functions
$l,u:\R\longrightarrow\R$ the bracket $[l,u]$ denotes the set of
functions~$f$ with $l\leq f\leq u$. For a set~$G$ of functions the
$L^2$-bracketing number~$N_{[\,]}(\eps,G)$ is the minimum number of
brackets $[l_i,u_i]$, satisfying $\E[(u_i(Z_1)\,-\,l_i(Z_1))^2]\leq
\eps^2$, that are needed to cover~$G$. The associated bracketing
integral is defined as \[ J_{[\,]}(\delta,G) \,=\, \int_0^\delta
\sqrt{ \log( N_{[\,]}(\eps,G) ) } \, d\eps. \] Furthermore, a
function $\overline{f}$ is called envelope function for~$G$, if
$|f|\leq \overline{f}$ holds for all~$f\in G$.

To apply Corollary~19.35 from \citeasnoun{vdV98}, we decompose $C_n$
in its real and imaginary parts,
\begin{eqnarray*}
\Re( C_n(u) ) & =
& n^{-1/2} \sum_{t=1}^n \left( \cos(u Z_t) \,-\, \E\cos(u Z_1)
\right), \\
\Im( C_n(u) ) & = & n^{-1/2} \sum_{t=1}^n \left( \sin(u Z_t) \,-\,
\E\sin(u Z_1) \right).
\end{eqnarray*}
Accordingly, we consider the
following class of functions:
\[
G_k \,=\, \left\{ \left. z \mapsto w(u) \tfrac{\partial^k}{\partial
u^k} \cos(uz) \right| u\in\R \right\} \cup \left\{ \left. z \mapsto
w(u) \tfrac{\partial^k}{\partial u^k} \sin(uz) \right| u\in\R
\right\}.
\]
An envelope function~$\overline{f}_k$ for~$G_k$ is given by
$\overline{f}_k=|x|^k$. Now we obtain from Corollary~19.35 in
\citeasnoun{vdV98} that
\begin{equation} \label{pl21.1}
\E \| C_n^{(k)} \|_{L_\infty(w)} \,\leq\, C \left\{
\E(\overline{f}_k(Z_1))^2 \,+\, J_{[\,]}(\sqrt{\E Z_1^{2k}}, G_k)
\right\}.
\end{equation}
Since $\E Z_1^{2k}<\infty$ it remains to bound the bracketing
integral on the right-hand side of (\ref{pl21.1}). Inspired by
\citeasnoun{Yuk85}, we proceed by setting, for every $\eps>0$,
\[
M \,:=\, M(\eps,k) \,:=\, \inf\left\{ m>0 \mid \E[Z_1^{2k}\,
\1_{\{|Z_1|>m\}}] \leq \eps^2 \right\}.
\]
Furthermore, we set, for grid points $u_j\in\R$ to be specified
below,
\begin{align*}
g_j^{\pm}(z) & = \left( w(u_j) \tfrac{\partial^k}{\partial u^k}
\cos(u_j z) \pm \eps |z|^k \right)
\1_{[-M,M]}(z) \,\pm\, \|w\|_\infty |z|^k\, \1_{[-M,M]^c}(z), \\
h_j^{\pm}(z) & = \left( w(u_j) \tfrac{\partial^k}{\partial u^k}
\sin(u_j z) \pm \eps |z|^k \right) \1_{[-M,M]}(z) \,\pm\,
\|w\|_\infty |z|^k\, \1_{[-M,M]^c}(z).
\end{align*}
We obtain for the width of the brackets that
\begin{align*}
\E\left[ \left( g_j^+(Z_1) \,-\, g_j^-(Z_1) \right)^2 \right] & \leq
\E\left[ 4 \eps^2 Z_1^{2k}\, \1_{[-M,M]}(Z_1)
\,+\, 4 \|w\|_\infty^2 Z_1^{2k} \1_{[-M,M]^c}(Z_1) \right] \\
& \le 4 \eps^2 \left( \E Z_1^{2k} \,+\, \|w\|_\infty^2 \right),
\end{align*}
and, analogously,
\[
\E\left[ \left( h_j^+(Z_1)
\,-\, h_j^-(Z_1) \right)^2 \right] \,\leq\, 4 \eps^2 \left( \E
Z_1^{2k} \,+\, \|w\|_\infty^2 \right).
\]

It remains to choose the grid points $u_j$ in such a way that the brackets
cover the set~$G_k$. We consider an arbitrary $u\in\R$ and any grid point
$u_j$. Then with the Lipschitz constant $\Lip(w)$ of the weight function $w$
\begin{eqnarray*}
\lefteqn{ \left| w(u) \tfrac{\partial^k}{\partial u^k} \cos(u z) \,-\,
w(u_j) \tfrac{\partial^k}{\partial u^k} \cos(u_j z) \right| } \\
& \leq & |z|^k \min\{ |u-u_j|(\Lip(w) \,+\, \|w\|_\infty |z|),
w(u)\,+\, w(u_j)\}.
\end{eqnarray*}
Therefore, the function $z \mapsto w(u) \tfrac{\partial^k}{\partial
u^k} \cos(uz)$ is contained in the bracket $[g_j^-,g_j^+]$ if
\[ \min\{ |u-u_j|(\Lip(w) \,+\,
\|w\|_\infty M), w(u) \,+\, w(u_j)\} \,\leq\, \eps.
\]
Consequently, we choose the grid points as
\[
u_j \,=\, j\eps/(\Lip(w) \,+\, \|w\|_\infty
M(\eps,k)),
\]
for $|j|\leq J(\eps)$, where $J(\eps)$ is the smallest integer such
that $u_{J(\eps)}$ is greater than or equal to
\[
U(\eps) \,=\, \inf\left\{ u>0\mid \sup_{v:\, |v|\geq
u} w(v) \leq \eps/2 \right\}.
\]
This yields the estimate $N_{[\,]}(\eps,G_k) \,\leq\, 2(2J(\eps) \,+\,
1)$. It follows from the generalized Markov inequality that
\[ M(\eps,k)
\,\leq\, \left( \E[|Z_1|^{2k+\gamma}]/\eps^2 \right)^{1/\gamma}.
\]
Now we obtain from the inequality
\[ J(\eps) \,\leq\,
2 U(\eps) (\Lip(w) \,+\, \|w\|_\infty M(\eps,k))/\eps \,+\, 1
\]
that $\log(N_{[\,]}(\eps,G_k))
=O(\log(J(\eps)))=O(\eps^{-(\delta+1/2)^{-1}}+\log(\eps^{-1-2/\gamma}))=
 O( \eps^{-\kappa} )$ for
$\kappa=(\delta+1/2)^{-1}<2$. This implies
\[ \int_0^\delta \sqrt{
\log(N_{[\,]}(\eps,G_k)) } \, d\eps \,<\, \infty,
\]
as required.\hfill\qed\\

\subsection{Proof of Theorem~\ref{T4.2}}\label{SPT4.2}

To simplify the notation, we use the abbreviations $\Psi_n(u)=\Psi(u;
\widehat{b}_n, \widehat{\nu}_{\sigma,n})$ and $\phi_n(u)=\exp(\Psi_n(u))$.

First of all, we obtain from the triangle inequality that
\begin{eqnarray} \label{pt42.1} d^{(2)}(\phi_n, \phi) & \leq &
d^{(2)}(\widehat{\phi}_n, \phi)
\,+\, d^{(2)}(\widehat{\phi}_n, \phi_n) \nonumber \\
& \leq & 2 d^{(2)}(\widehat{\phi}_n, \phi) \,+\, \delta_n.
\end{eqnarray}

\noindent {\em Proof for $\widehat{b}_n$}

\noindent We have that $\phi'(0)=ib$ and $\phi_n'(0)=i\widehat{b}_n$.
Therefore, we obtain from (\ref{pt42.1}) and Theorem~\ref{T4.1} that
\begin{align*}
\E_{b,\nu_\sigma} |\widehat{b}_n \,-\, b|
& = \E_{b,\nu_\sigma} |\phi_n'(0) \,-\, \phi'(0)| \\
& \le \E_{b,\nu_\sigma} d^{(2)}(\phi_n, \phi) \\
& \le 2 \E_{b,\nu_\sigma} d^{(2)}(\widehat{\phi}_n, \phi) \,+\,
\delta_n \,=\, O( n^{-1/2} ).
\end{align*}

\noindent {\em Proof for $\widehat{\nu}_{\sigma,n}$}

\noindent We consider the following set of ``unfavorable'' events:
\[ A_n
\,:=\, \left\{ \widehat{\nu}_{\sigma,n}(\R) \,>\, \nu_\sigma(\R)
\,+\, 1 \right\} \cup \left\{ |\widehat{b}_n| \,>\, |b| \,+\, 1
\right\}.
\]
{}From $\phi'(0)^2-\phi''(0)=\nu_\sigma(\R)$ and the analogous formula for
$\phi_n$ it follows that
\begin{align}
\abs{\widehat{\nu}_{\sigma,n}(\R)-\nu_\sigma(\R)}
&=\abs{(\phi'(0)^2-\phi_n'(0)^2)-(\phi''(0)
-\phi_n''(0))}\nonumber\\
&\le (2\abs{\phi'(0)}+d^{(2)}( \phi, \phi_n)+1)d^{(2)}(\phi, \phi_n),
\label{Rel-nu}
\end{align}
Consequently, the (generalized) Markov inequality yields
\begin{eqnarray*}
P_{b,\nu_\sigma}( A_n ) & \leq & \E_{b,\nu_\sigma} [
|\widehat{\nu}_{\sigma,n}(\R)
\,-\, \nu_\sigma(\R)|\wedge 1] \,+\, \E_{b,\nu_\sigma} [|\widehat{b}_n \,-\, b|] \\
& \le & \E_{b,\nu_\sigma} \left[ \left(\left(2\abs{b} +d^{(2)}( \phi, \phi_n
)+1\right)d^{(2)}( \phi, \phi_n )\right) \wedge 1\right]
\,+\, \E_{b,\nu_\sigma} \left[d^{(2)}( \phi, \phi_n ) \right]\\
& \le & \E_{b,\nu_\sigma} \left[ (4\abs{b}+2)d^{(2)}( \phi, \phi_n ) \right]
\,+\, \E_{b,\nu_\sigma} \left[d^{(2)}( \phi, \phi_n ) \right]\\
& \leq & (4\abs{b}+3) \left(\E_{b,\nu_\sigma}[d^{(2)}(\widehat{\phi}_n, \phi)]
\,+\, \delta_n\right) \,=\, O( n^{-1/2} ),
\end{eqnarray*}
which implies that
\begin{eqnarray} \label{pt42.2}
\lefteqn{ \1_{A_n} \cdot \sup_{f\in{F}_{s}} \left| \int f\,
d\widehat{\nu}_{\sigma,n} \,-\, \int f\, d\nu_\sigma
\right| } \nonumber \\
& \leq &  \1_{A_n} \cdot \sup_{f\in{F}_{s}} \|f\|_\infty \cdot \left(
\widehat{\nu}_{\sigma,n}(\R) \,+\, \nu_\sigma(\R)
\right) \nonumber \\
& \leq &  2 \; \nu_\sigma(\R) \; \1_{A_n} \,+\,
(2\abs{\phi'(0)}+d^{(2)}( \phi, \phi_n)+1)d^{(2)}(\phi, \phi_n) \nonumber \\
& = & O_{P_{b,\nu_\sigma}}( n^{-1/2} ).
\end{eqnarray}

It remains to analyse the loss under $A_n^c$. It follows from
Parseval's identity that
\begin{eqnarray} \label{pt42.3}
\lefteqn{ \left| \int f\,
d\widehat{\nu}_{\sigma,n}
\,-\, \int f\, d\nu_\sigma \right| } \nonumber \\
& = & \frac{1}{2\pi} \int_{-\infty}^\infty \FT f(u) \overline{ (\FT
\widehat{\nu}_{\sigma,n}(u) \,-\,
\FT \nu_\sigma(u)) } \, du \nonumber \\
& = & \frac{1}{2\pi} \int_{-\infty}^\infty \FT f(u) \overline{ \left\{ \left(
\big( \frac{\phi_n'(u)}{\phi_n(u)} \big)^2 \,-\, \big( \frac{\phi'(u)}{\phi(u)}
\big)^2 \right) \;-\; \left( \frac{\phi_n^{''}(u)}{\phi_n(u)} \,-\,
\frac{\phi^{''}(u)}{\phi(u)} \right) \right\} } \, du. \quad
\end{eqnarray}
The differences occurring in the integrand on the right-hand side of
(\ref{pt42.3}) can be estimated using $\phi'/\phi=\Psi'$,
$\phi_n'/\phi_n=\Psi_n'$:
\begin{align}
\label{pt42.4}
&\left| \left( \frac{\phi_n'(u)}{\phi_n(u)} \right)^2 \,-\, \left(
\frac{\phi'(u)}{\phi(u)} \right)^2 \right|  =  \left|
\frac{\phi_n'(u)}{\phi_n(u)} \,-\, \frac{\phi'(u)}{\phi(u)} \right|
\; | \Psi_n'(u) \,+\, \Psi'(u) | \nonumber \\
&\qquad\qquad\qquad\qquad \le \left\{ \left| \frac{ \phi_n(u) \,-\, \phi(u)
}{\phi(u)} \right| |\Psi_n'(u)| \,+\, \left| \frac{ \phi_n'(u) \,-\, \phi'(u)
}{\phi(u)} \right| \right\} \; | \Psi_n'(u) \,+\, \Psi'(u) |
\end{align}
and
\begin{eqnarray} \label{pt42.5}
\left| \frac{\phi_n''(u)}{\phi_n(u)} \,-\, \frac{\phi''(u)}{\phi(u)}
\right| & \leq & \left| \frac{ \phi_n(u) \,-\, \phi(u) }{\phi(u)} \right| \;
\left| \frac{\phi_n''(u)}{\phi_n(u)} \right| \,+\, \left| \frac{
\phi_n''(u) \,-\, \phi''(u) }{\phi(u)} \right|
\nonumber \\
& = & \left| \frac{ \phi_n(u) \,-\, \phi(u) }{\phi(u)} \right| \; \left|
\Psi_n''(u) \,+\, \left( \Psi_n'(u) \right)^2 \right| \,+\, \left| \frac{
\phi_n''(u) \,-\, \phi''(u) }{\phi(u)} \right|.
\end{eqnarray}
Note that the following estimates hold true under $A_n^c$:
\begin{eqnarray} \label{pt42.6}
|\Psi_n'(u)| & \leq & |\widehat{b}_n| \,+\, |u|
\widehat{\nu}_{\sigma,n}(\R)
\,\leq\, |b| \,+\, 1 \,+\, |u|(\nu_\sigma(\R) \,+\, 1), \\
\label{pt42.7} |\Psi_n''(u)| & \leq & |\FT
\widehat{\nu}_{\sigma,n}(u)| \,\leq\, \widehat{\nu}_{\sigma,n}(\R)
\,\leq\, \nu_\sigma(\R) \,+\, 1.
\end{eqnarray}
Hence, we obtain
from (\ref{pt42.3}) to (\ref{pt42.7}) and the trivial estimate $|\FT
\widehat{\nu}_{\sigma,n}(u) \,-\, \FT \nu_\sigma(u)| \leq
\widehat{\nu}_{\sigma,n}(\R) \,+\, \nu_\sigma(\R)$ that under
$A_n^c$, with some constant $C>0$,
\begin{eqnarray*}
\lefteqn{ \left| \int f \, d\widehat{\nu}_{\sigma,n}
\,-\, \int f \, d\nu_\sigma \right| } \nonumber \\
& \leq & C \; \int_{-\infty}^\infty |\FT f(u)| \; \left( \frac{
|\phi_n(u)-\phi(u)| \,+\, |\phi_n'(u)-\phi'(u)| \,+\,
|\phi_n''(u)-\phi''(u)| }{ |\phi(u)| }
(1 \,+\, |u|)^2 \;\wedge\; 1 \right) \, du \nonumber \\
& \leq & C \; \int_{-\infty}^\infty (1\,+\,|u|)^s |\FT f(u)|\, du \;\cdot\;
\sup_{u\in\R} \left\{ (1 \,+\, |u|)^{-s} \left( \frac{ (1\,+\,|u|)^2 \;
d^{(2)}(\phi_n,\phi) }{ w(u) |\phi(u)| }
\;\wedge\; 1 \right) \right\} \label{pt42.8}\\
& \leq & C \; \sup_{u\geq 0} \left\{ (1 \,+\, u)^{-s} \left( \frac{
(1\,+\,u)^2 \; n^{-1/2} }{ w(u) |\phi(u)| } \;\wedge\; 1 \right)
\right\} \; \left( n^{1/2} d^{(2)}(\phi_n,\phi) \,+\, 1 \right).
\nonumber
\end{eqnarray*}
By monotonicity of $(1+u)^{-s}$ we can replace the supremum over $[0,\infty)$
by the supremum over $[0,U_n]$ and we arrive at
\begin{eqnarray}
\label{pt42.9} \lefteqn{ \E_{b,\nu_\sigma} \left[ \1_{A_n^c} \cdot \sup_{f\in
F_{s}} \left| \int f \, d\widehat{\nu}_{\sigma,n}
\,-\, \int f \, d\nu_\sigma \right| \right] } \nonumber \\
& = & O\left( \sup_{u\in [0,U_n]} \left\{ (1\,+\,u)^{-s} \left( \frac{
(1\,+\,u)^2 n^{-1/2} }{ w(u) |\phi(u)| } \wedge 1 \right) \right\} \right).
\end{eqnarray}
Together with the bound (\ref{pt42.2}) on the set $A_n$ this yields
the asserted general estimate. Tracing back the constants, we see that they depend continuously on
$\abs{b}$ and $\nu_\sigma(\R)$.\\

\noindent {\em Proof of the rate results (a), (b)}

\noindent {\bf (a)} Under the condition $\log|\phi(u)|= -\sigma^2 u^2/2(1+o(u))$
we have $U_n\asymp \sqrt{\log n}$ and we obtain the rate $U_n^{-s}=(\log
n)^{-s/2}$.

\noindent {\bf (b)} If $|\phi(u)|\geq C e^{-\alpha u}$, then we have
$U_n\asymp \log n$ and we obtain the rate $U_n^{-s}=(\log n)^{-s}$.\\

\noindent{\em Proof of the rate result (c)}

The same reasoning as for cases (a) and (b) would only yield the
rate $((\log n)^{1/2+\delta}n^{1/2})^{-s/(\beta+2)}$ for $s\in
(0,\beta+2]$ and the parametric rate for $s>\beta+2$. In the
polynomial case (c), though, better estimates for $\abs{\Psi_n'(u)}$
hold, i.e. we can improve upon (\ref{pt42.6}). First, we formulate
and prove a lemma for $\abs{\Psi'(u)}$.

\begin{lem}\label{L6.1}
If a L\'evy process with a finite first moment has a characteristic
function (at time $t=1$) satisfying $| \phi(u) | \,\ge\, C(1 \,+\,
|u|)^{-\beta}$ for some $\beta\ge 0$, $C>0$ and all $u\in\R$, then
$\int_{[-1,+1]}\abs{x}^\alpha\nu(dx)$ is finite for all $\alpha>0$
and the derivative of its characteristic exponent is uniformly
bounded:
\[ \sup_{u\in\R}\abs{\Psi'(u)}<\infty.\]
\end{lem}

\begin{proof}[Proof of Lemma~\ref{L6.1}]
Since we have necessarily $\sigma^2=0$ in the L\'evy-Khinchine
characteristic as well as $\int_{[-1,1]^c}|x|\,\nu(dx)<\infty$ from
the first moment condition, the additional property
$\int_{[-1,+1]}\abs{x}\nu(dx)<\infty$ implies
\[
\sup_{u\in\R}\abs{\Psi'(u)}
\,=\, \sup_{u\in\R}\babs{ib+\int
(e^{iux}-1)ix\,\nu(dx)}
\,\le\, \abs{b}+2\int\abs{x}\,\nu(dx)
\,<\, \infty.
\]
It therefore remains to prove the first result for any $\alpha>0$.
We obtain with $c:=\min_{u\in [1,2]}(1-\cos(u))>0$:
\begin{align*}
\int_{[-1,+1]} \abs{x}^\alpha \nu(dx) &\le \sum_{n=1}^\infty
\int_{\{x:\,2^{-n}\le\abs{x}\le 2^{-n+1}\}} \abs{x}^{\alpha}
\nu(dx)\\
&\le  \sum_{n=1}^\infty 2^{-\alpha (n-1)}
\int_{\{x:\,2^{-n}\le\abs{x}\le 2^{-n+1}\}}
c^{-1}(1-\cos(2^nx))\, \nu(dx)\\
&\le c^{-1}\sum_{n=1}^\infty 2^{-\alpha (n-1)}\Re(-\Psi(2^n))\\
&\le c^{-1}\sum_{n=1}^\infty 2^{-\alpha (n-1)}
\left(\log(C^{-1})+\beta\log(1 \,+\, 2^n)\right).
\end{align*}
This latter series is obviously finite.
\end{proof}

Resuming the proof for case (c), we remark that $\abs{\phi(u)}\ge
C(1+\abs{u})^{-\beta}$ implies for any $U>0$
\begin{align*}
P_{b,\nu_\sigma}&\left(\exists u\in [-U,U]:\: \abs{\phi_n(u)} < \frac{C}{2}(1
\,+\, |u|)^{-\beta}\right)\\
&\le P_{b,\nu_\sigma}\left(\sup_{\abs{u}\le
U}\abs{\phi_n(u)-\phi(u)}(1+\abs{u})^\beta\ge C/2\right)\\
&\le \frac{2}{C}
\E[\norm{\phi_n-\phi}_{L^\infty(w)}]w(U)^{-1}(1+U)^\beta
=O(n^{-1/2}w(U)^{-1}(1+U)^\beta).
\end{align*}
Consequently, for $U_n\to\infty$ with $w(U_n)^{-1}U_n^\beta=o(n^{1/2})$ we have
\begin{equation}\label{BoundPhi-n}
\lim_{n\to\infty}P_{b,\nu_\sigma}\left(\forall u\in [-U_n,U_n]:\:
\abs{\phi_n(u)} \ge \frac{C}{2}(1 \,+\, |u|)^{-\beta}\right)=1 ;
\end{equation}
in the sequel we shall work with $U_n=n^{1/(2\beta)}(\log
n)^{-(1/2+2\delta)/\beta}$. Theorem~\ref{T4.1}, Lemma~\ref{L6.1}
and Equation~\eqref{BoundPhi-n}
then yield
\begin{eqnarray}
\label{p43.1} \sup_{\abs{u}\le U_n}|\Psi_n'(u)-\Psi'(u)| & \le &
\sup_{\abs{u}\le
U_n}\left\{\frac{|\phi_n'(u)-\phi'(u)|}{|\phi_n(u)|}+|\Psi'(u)|\frac{|\phi(u)-\phi_n(u)|}{|\phi_n(u)|}\right\}
\nonumber \\
& = & O_P(n^{-1/2})w(U_n)^{-1}\frac{2}{C}(1+|U_n|)^\beta.
\end{eqnarray}

Together with Estimate (\ref{p43.1}) and again Lemma \ref{L6.1} we have thus
established for $n\to\infty$
\begin{equation}\label{p43.2}
\sup_{\abs{u}\le U_n}
\abs{\Psi_n'(u)} \,=\, O_P\left( 1 \,+\, n^{-1/2}w(U_n)^{-1}|U_n|^\beta \right) \,=\, O_P(1).
\end{equation}
We therefore get instead of (\ref{pt42.8}) the estimate
\begin{eqnarray*}
\lefteqn{ \sup_{f\in F_{s}} \left| \int f \, d\widehat{\nu}_{\sigma,n}
\,-\, \int f \, d\nu_\sigma \right| }  \\
& = &  \sup_{u\in\R} \left\{ (1 \,+\, |u|)^{-s} \left( \frac{
(O_P(1)+u^2\,\1_{\{\abs{u}\ge U_n\}})n^{-1/2} }{ w(u) |\phi(u)| } \;\wedge\; 1
\right) \right\} \; O_P\left(
n^{1/2} d^{(2)}(\phi_n,\phi) \,+\, 1 \right)\\
& = &  \sup_{u\in\R} \left\{ (1 \,+\, |u|)^{-s} \left( \frac{
(O_P(1)+u^2\,\1_{\{\abs{u}\ge U_n\}})n^{-1/2} }{
(\log(e+\abs{u}))^{-1/2-\delta}(1+\abs{u})^{-\beta} } \;\wedge\; 1 \right)
\right\}\;O_P(1) .
\end{eqnarray*}
For $s\le\beta$ the right-hand side is of order $O_P(U_n^{-s})$ and we
obtain
\[
\sup_{f\in F_{s}} \left| \int f \, d\widehat{\nu}_{\sigma,n} \,-\, \int f \,
d\nu_\sigma \right| =O_P\left(n^{-s/2\beta}(\log n)^{s(1/2+2\delta)/\beta}\right),
\]
while for $s>\beta$ the parametric rate $O_P(n^{-1/2})$
follows.\hfill\qed\\

\subsection{Proof of Theorem \ref{T4.4}}\label{SPT4.4}

The lower bound will be established by looking at a decision problem between
two local alternatives, see e.g. \cit{KT93} for the general idea. For
$\gamma>0$ and $\beta>0$ consider the bilateral Gamma distribution which is
obtained as the law of $X-Y$ where $X$ and $Y$ are independent and both
$\Gamma(\gamma,\beta/2)$-distributed. This bilateral Gamma distribution is
infinitely divisible with the following characteristic function and L\'evy
triplet:
\[
\phi_\Gamma(u):=\left(1+\gamma^{-2}u^2\right)^{-\beta/2},\quad
b_\Gamma=0,\,\sigma_\Gamma=0,\,\nu_\Gamma(dx):=\beta
\abs{x}^{-1}e^{-\gamma\abs{x}}dx.
\]
Its density $f_\Gamma$ satisfies $f_\Gamma(x)\ge c e^{-\gamma
\abs{x}}$ for some $c>0$ \cite{KT06}. For $\sigma\ge 0$ consider the
infinitely divisible distribution with characteristic function
\begin{equation}\label{P51.1}
\phi_0(u):=\phi_\Gamma(u)e^{iub-\sigma^2 u^2/2},
\end{equation}
which has a density $f_0$ that is a convolution of $f_\Gamma$ with a
normal density and therefore still satisfies $f_0(x)\ge
c\,e^{-\gamma\abs{x}}$ with some $c>0$. The corresponding L\'evy
density satisfies $\nu_0=\nu_\Gamma$.

Let us further introduce for $K>0$ and $\rho>0$
\[
\mu_K(x):=e^{-x^2/(2\rho^2)}\sin(Kx).
\]
For any $\beta>0$ and $\gamma>0$ we can choose $\rho$ sufficiently
small such that $\nu_0(x)+\mu_K(x)\ge 0$ holds for all $K>0$. In
this case the following characteristic function also generates an
infinitely divisible distribution:
\[
\phi_K(u) \,:=\, \phi_0(u)\exp\left(\int_{\R}
(e^{iux}-1)\,\mu_K(dx)\right) \,=\, \phi_0(u)\exp(\FT\mu_K(u)).
\]
Using the fact that
$\sin(Kx)\sin(ux)=(\cos((K-u)x)-\cos((K+u)x))/2$
we obtain the following explicit calculation of the Fourier
transform of $\mu_K$:
\begin{eqnarray*}
\FT\mu_K(u) 
& = & i\int_{-\infty}^\infty
e^{-x^2/(2\rho^2)}\sin(Kx)\sin(ux)\,dx \\
& = & i\int_0^\infty e^{-x^2/(2\rho^2)} \cos((K-u)x) \, dx
\,-\, i\int_0^\infty e^{-x^2/(2\rho^2)} \cos((K+u)x) \, dx \\
& = & i\rho\sqrt{\pi/2}\; (e^{-\rho^2(K-u)^2/2}-e^{-\rho^2(K+u)^2/2}).
\end{eqnarray*}
Note that $\phi_K$ has the same decay behavior as $\phi_0$ due to
$\lim_{\abs{u}\to\infty} \FT\mu_K(u)=0$. Therefore $\nu_{0,\sigma}$
and $\nu_{K,\sigma}$ lie in the class ${\mathcal A}(C,\sigma)$
($\sigma>0$) or ${\mathcal C}(C,\bar{C},\beta)$ ($\sigma=0$),
respectively, provided $C,\bar{C}$ are large enough.

Let us now estimate the $\chi^2$-distance between the distributions
with characteristic functions $\phi_K$ and $\phi_0$:
\begin{eqnarray}
\chi^2(f_K,f_0)
& := & \int_{-\infty}^\infty \frac{(f_K(x)-f_0(x))^2}{f_0(x)}\,dx \nonumber\\
& \le & c^{-1} \int_{-\infty}^\infty
\left(e^{\gamma \abs{x}/2}f_K(x)-e^{\gamma \abs{x}/2}f_0(x)\right)^2\,dx \nonumber\\
& \le & c^{-1} \left\{ \int_{-\infty}^\infty \left(e^{\gamma
x/2}f_K(x)-e^{\gamma x/2}f_0(x)\right)^2\,dx \right. \\
& & {} \qquad \,+\, \left.
\int_{-\infty}^\infty \left(e^{-\gamma x/2}f_K(x)-e^{-\gamma
x/2}f_0(x)\right)^2\,dx \right\}. \nonumber 
\label{T51.2}
\end{eqnarray}
For functions $g$ whose Fourier
transform can be extended holomorphically to complex values $z$ with
$\abs{\Im(z)}<\gamma$ we have:
\[ \FT\left(e^{\pm\gamma x/2}g(x)\right)(u)=\int
g(x)e^{(iu\pm\gamma/2)x}dx=\FT g(u\pm(-i)\gamma/2).
\]
Using this identity in Plancherel's formula and then the estimate
$\abs{e^z-1}\le \abs{z}e^{\abs{\Re(z)}}$, $z\in\C$, together with
$\abs{\FT\mu_K(u)}\le \norm{\mu_K}_{L^1}$, we continue from
(\ref{T51.2}):
\begin{eqnarray*}
\lefteqn{ \chi^2(f_K,f_0) } \\
& \le & \frac{c^{-1}}{2\pi} \int_{-\infty}^\infty
(\abs{\phi_K(u-i\gamma/2)-\phi_0(u-i\gamma/2)}^2+\abs{\phi_K(u+i\gamma/2)-\phi_0(u+i\gamma/2)}^2)\,du\\
& = & \frac{c^{-1}}{2\pi}\int_{-\infty}^\infty e^{-\sigma^2u^2}
\babs{\frac{3}{4}+\frac{u^2}{\gamma^2}+\frac{iu}{\gamma}}^{-\beta}
\left( \abs{e^{\FT\mu_K(u-i\gamma/2)}-1}^2
\,+\,  \abs{e^{\FT\mu_K(u+i\gamma/2)}-1}^2 \right)\,du\\
& \le & \frac{e^{2\norm{\mu_K}_{L^1}}}{2c\pi} \int_{-\infty}^\infty
e^{-\sigma^2u^2}\left(\frac{3}{4}+\frac{u^2}{\gamma^2}\right)^{-\beta}
\left( \abs{\FT\mu_K(u-i\gamma/2)}^2  \,+\, \abs{\FT\mu_K(u+i\gamma/2)}^2 \right) \,du\\
& = & \frac{e^{2\norm{\mu_K}_{L^1}}\rho^2}{4c\pi}
\int_{-\infty}^\infty
e^{-\sigma^2u^2}\left(\frac{3}{4}+\frac{u^2}{\gamma^2}\right)^{-\beta}
\left(e^{-\rho^2(u-K)^2/2}-e^{-\rho^2(u+K)^2/2}\right)^2\,du.
\end{eqnarray*}
The last line is for $K\to\infty$ of order $\int_{-\infty}^\infty
e^{-\sigma^2u^2} (1+u^2)^{-\beta} (e^{-\rho^2(u-K)^2}+e^{-\rho^2(u+K)^2})\,du$.
In the case $\sigma=0$ (polynomial decay) this gives the order $K^{-2\beta}$,
whereas for $\sigma>0$ (Gaussian part) the order is $e^{-\sigma^2K^2(1+o(1))}$.

For $n$ observations the distributions do not separate provided $K^{-2\beta}\asymp
n^{-1}$ ($\sigma=0$) and $e^{-\sigma^2K^2(1+o(1))}\asymp n^{-1}$ ($\sigma>0$),
respectively. Consequently, when choosing $K_n\asymp n^{1/2\beta}$
($\sigma=0$), respectively $K_n=c \sqrt{\log(n)}$ with $c>0$ sufficiently large
($\sigma>0$), this closeness of the distributions implies \cite{KT93} that for any sequence of estimators
$(\widehat{\nu}_{\sigma,n})_n$ we have
\[
\liminf_{n\to\infty}\Big\{P_0\big(\ell_s(\widehat{\nu}_{\sigma,n},\nu_{0,\sigma})\ge\ell_s(\nu_{K_n,\sigma},\nu_{0,n})/2\big)+
P_{K_n}\big(\ell_s(\widehat{\nu}_{\sigma,n},\nu_{0,\sigma})\ge\ell_s(\nu_{K_n,\sigma},\nu_{0,n})/2\big)\Big\}>0.
\]
It remains to consider the loss $\ell_s$ between the alternatives. Using the
formula $\FT(x^2e^{-x^2/(2\rho^2)})(u)=\rho^3(1-\rho^2u^2)e^{-\rho^2u^2/2}$, we
calculate:
\begin{align*}
\ell_s(\nu_{K,\sigma},\nu_{0,\sigma})&=\sup_{f\in
F_{s}}\babs{\int_{-\infty}^\infty f(x)x^2e^{-x^2/2\rho^2}\sin(Kx)\,dx}\\
&=\frac{1}{2\pi}\sup_{f\in
F_{s}}\babs{\Im((\FT f \ast \FT(x^2e^{-x^2/2\rho^2}))(K))}\\
&=\frac{1}{2\pi}\sup_{f\in
F_{s}}\babs{\int_{-\infty}^\infty \Im(\FT f(x))\rho^3(1-\rho^2(K-u)^2) e^{-\rho^2(K-u)^2/2}\,du}\\
&=\frac{1}{2\pi}\rho^3\sup_{u\in\R}\left\{(1+\abs{u})^{-s}\abs{1-\rho^2(K-u)^2}e^{-\rho^2(K-u)^2/2}\right\}\\
&\asymp K^{-s}.
\end{align*}
Setting
$\eps:=\liminf_{n\to\infty}K_n^s\ell_s(\nu_{K,\sigma},\nu_{0,\sigma})/2>0$, we
have thus shown
\[
\liminf_{n\to\infty}\sup_{\nu_\sigma}P_{b,\nu_\sigma}(K_n^s\ell_s(\widehat{\nu}_{\sigma,n},\nu_{\sigma})\ge\eps)>0.
\]
For $\sigma=0$ (polynomial decay) this gives the desired lower bound
$K_n^{-s}=n^{-s/(2\beta)}$ for any $\beta>0$ and for $s\le\beta$.
For $s>\beta$ a standard parametric argument shows that the minimax
rate is never faster than $n^{-1/2}$. For $\sigma>0$ (Gaussian part)
we obtain the lower bound $K_n^{-s}=(\log n)^{-s/2}$, which matches
exactly the upper bound.

In the case (b), i.e. where $|\phi(u)|\ge C e^{-\alpha\abs{u}}$,
we consider instead of (\ref{P51.1})
\[ \phi_0(u)=\phi_\Gamma(u)\phi_\alpha(u),\]
where $\phi_\alpha$ is an infinitely divisible characteristic
function with $\abs{\phi(u)}\asymp e^{-\alpha\abs{u}}$ such that the
corresponding density function $f_\alpha$ has faster exponential
decay than $f_0$.  For example, a tempered stable law \cite[Prop.
4.2]{CT04} with $\nu(dx)=\alpha\abs{x}^{-2}e^{-\abs{\lambda}x}dx$
and $\lambda>0$ sufficiently large meets these requirements. The
remaining steps of the proof are exactly the same, just replace
$e^{-\sigma^2u^2/2}$ by $e^{-\alpha\abs{u}}$.
\hfill\qed\\

\subsection{Proof of Proposition \ref{P5.1}}\label{SPP5.1}

Note first that $\E_{b,\nu_\sigma}\abs{\widetilde b_n-b}^2=O(n^{-1})$ follows
directly from $\E X_1^2<\infty$.

To prove the result for the jump measure, we distinguish between two
cases. We set
$\widehat{\Psi}_n'(u)=\widehat{\phi}_n'(u)/\widehat{\phi}_n(u)$ and
$\widehat{\Psi}_n''(u)=\widehat{\phi}_n''(u)/\widehat{\phi}_n(u)
-(\widehat{\phi}_n'(u)/\widehat{\phi}_n(u))^2$.\\

\noindent{\em Case~1: $|\phi(u)|\geq 2\kappa n^{-1/2}$}

It follows from (\ref{pt42.4}) and (\ref{pt42.5}) that
\begin{eqnarray}
\label{pp51.1}
\lefteqn{ \left| \FT \widetilde{\nu}_{\sigma,n}(u)
\,-\, \FT\nu_\sigma(u) \right| } \nonumber \\
& \leq & \left\{ \left| \frac{\widehat{\phi}_n(u)-\phi(u)}{\phi(u)} \right|
|\widehat{\Psi}_n'(u)|
\,+\, \left| \frac{\widehat{\phi}_n'(u)-\phi'(u)}{\phi(u)} \right|
\left| \widehat{\Psi}_n'(u) \,+\, \Psi'(u) \right| \right\}
\1_{\{\abs{\hat{\phi}_n(u)}\geq\kappa n^{-1/2}\}} \nonumber \\
& & {} + \left\{ \left| \frac{\widehat{\phi}_n(u)-\phi(u)}{\phi(u)} \right|
\left|\widehat{\Psi}_n''(u) \,+\, (\widehat{\Psi}_n'(u))^2 \right|
\,+\, \left| \frac{\widehat{\phi}_n''(u)-\phi''(u)}{\phi(u)} \right| \right\}
\1_{\{\abs{\hat{\phi}_n(u)}\geq\kappa n^{-1/2}\}} \\
& & {} + \left| \FT \nu_\sigma(u) \right|
\1_{\{\abs{\hat{\phi}_n(u)}<\kappa n^{-1/2}\}} \nonumber \\
& = & T_{n,1} \,+\, T_{n,2} \,+\, T_{n,3}, \nonumber
\end{eqnarray}
say.

We obtain from the inequality
\begin{displaymath}
\frac{ \1_{\{\abs{\hat{\phi}_n(u)}\geq\kappa n^{-1/2}\}} }{ \widehat{\phi}_n(u) }
\,\leq\, \frac{1}{|\phi(u)|}
\,+\, \frac{ |\widehat{\phi}_n(u)\,-\,\phi(u)| }{ \kappa n^{-1/2} |\phi(u)| }
\end{displaymath}
that
\begin{equation}
\label{pp51.2} \E\left[ |\widehat{\phi}_n(u)|^{-p}\,
\1_{\{\abs{\hat{\phi}_n(u)}\geq\kappa n^{-1/2}\}} \right] \,=\,
O\left( |\phi(u)|^{-p} \right)
\end{equation}
holds for all $p\in\N$. This implies, by
$\widehat{\Psi}_n'(u)=(\widehat{\phi}_n'(u)-\phi'(u))/\widehat{\phi}_n(u)
+\Psi'(u) \phi(u)/\widehat{\phi}_n(u)$, that
\begin{equation}
\label{pp51.3}
\E\left[ \left| \widehat{\Psi}_n'(u) \right|^p
\1_{\{\abs{\hat{\phi}_n(u)}\geq\kappa n^{-1/2}\}} \right]
\,=\, O\left( (1 \,+\, |\Psi'(u)|)^p \right).
\end{equation}
Therefore, we obtain that
\begin{equation}
\label{pp51.4} \E T_{n,1} \,=\, O\left( \frac{n^{-1/2}}{|\phi(u)|}
\left( 1\,+\,|\Psi'(u)| \right)^2 \right).
\end{equation}
Since
\begin{eqnarray*}
\widehat{\Psi}_n''(u)
& = & \frac{ \widehat{\phi}_n''(u) }{ \widehat{\phi}_n(u) }
\,-\, \left( \widehat{\Psi}_n'(u) \right)^2 \\
& = & \frac{ \widehat{\phi}_n''(u) - \phi''(u) }{ \widehat{\phi}_n(u) }
\,+\, \left( \Psi''(u) + \left( \Psi'(u) \right)^2 \right)
\frac{ \phi(u) }{ \widehat{\phi}_n(u) } \,-\, \left( \widehat{\Psi}_n'(u) \right)^2
\end{eqnarray*}
we obtain, in conjunction with (\ref{pp51.2}) and (\ref{pp51.3}), that
\begin{displaymath}
\E \left[ \left| \widehat{\Psi}_n''(u) \right|^2
\1_{\{\abs{\hat{\phi}_n(u)}\geq\kappa n^{-1/2}\}} \right]
\,=\, O\left( \left( 1\,+\,|\Psi'(u)| \right)^2 \right).
\end{displaymath}
We conclude that
\begin{equation}
\label{pp51.5} \E T_{n,2} \,=\, O\left( \frac{n^{-1/2}}{|\phi(u)|}
\left( 1\,+\,|\Psi'(u)| \right)^2 \right).
\end{equation}
Finally, it follows from Hoeffding's inequality for bounded random variables that
\begin{eqnarray*}
P\left( |\widehat{\phi}_n(u)| \,<\, \kappa n^{-1/2} \right)
& \leq & P\left( |\widehat{\phi}_n(u) - \phi(u)|
\,>\, |\phi(u)| \,-\, \kappa n^{-1/2} \right) \\
& \leq & P\left( |\widehat{\phi}_n(u) - \phi(u)|
\,>\, |\phi(u)|/2 \right) \\
& \leq & \exp( -c \; n \; |\phi(u)|^2 ),
\end{eqnarray*}
for some $c>0$. This yields that $P(|\widehat{\phi}_n(u)|<\kappa
n^{-1/2})=O(n^{-1/2} |\phi(u)|^{-1})$, and therefore
\begin{equation}
\label{pp51.6} \E T_{n,3} \,=\, O\left( \frac{n^{-1/2}}{|\phi(u)|}
\right).
\end{equation}
Equations (\ref{pp51.1}), (\ref{pp51.4}), (\ref{pp51.5}), and
(\ref{pp51.6}) yield the desired bound in the case $|\phi(u)|\geq
2\kappa
n^{-1/2}$.\\

\noindent{\em Case~2: $|\phi(u)|< 2\kappa n^{-1/2}$}

In contrast to Case~1, this time we use the following decomposition:
\begin{eqnarray}
\label{pp51.11}
\lefteqn{ \left| \FT \widetilde{\nu}_{\sigma,n}(u)
\,-\, \FT\nu_\sigma(u) \right| } \nonumber \\
& \leq & \left\{ \left| \frac{\widehat{\phi}_n(u)-\phi(u)}{\widehat{\phi}_n(u)}
\right| |\Psi'(u)|
\,+\, \left| \frac{\widehat{\phi}_n'(u)-\phi'(u)}{\widehat{\phi}_n(u)} \right|
\left| \Psi'(u) \,+\, \widehat{\Psi}_n'(u) \right| \right\}
\1_{\{\abs{\hat{\phi}_n(u)}\geq\kappa n^{-1/2}\}} \nonumber \\
& & {} + \left\{ \left| \frac{\widehat{\phi}_n(u)-\phi(u)}{\widehat{\phi}_n(u)} \right|
\left|\Psi''(u) \,+\, (\Psi'(u))^2 \right|
\,+\, \left| \frac{\widehat{\phi}_n''(u)-\phi''(u)}{\widehat{\phi}_n(u)} \right| \right\}
\1_{\{\abs{\hat{\phi}_n(u)}\geq\kappa n^{-1/2}\}} \\
& & {} + \left| \FT \nu_\sigma(u) \right|
\1_{\{\abs{\hat{\phi}_n(u)}<\kappa n^{-1/2}\}}. \nonumber
\end{eqnarray}
Taking into account that $\Psi''$ is bounded and using again (\ref{pp51.3})
as well as the trivial estimate
$|\FT\nu_\sigma(u)|\leq\nu_\sigma(\R)<\infty$ we obtain that
\begin{displaymath}
\E \left| \FT \widetilde{\nu}_{\sigma,n}(u)
\,-\, \FT\nu_\sigma(u) \right|
\,=\, O\left( \left( 1\,+\,|\Psi'(u)| \right)^2 \right),
\end{displaymath}
as required.
\hfill\qed\\

\begin{ack}
We thank Peter Tankov for the idea how to prove Lemma \ref{L6.1} and
Shota Gugushvili for useful discussions and hints.
\end{ack}

\bibliographystyle{dcu}
\thebibliography{99}

\harvarditem{Basawa and Brockwell}{1982}{BB82} {\sc Basawa, I.V.}
and {\sc Brockwell, P.J.} (1982). Non-parametric estimation for
non-decreasing L\'evy processes, {\sl J. R. Statist. Soc. B } 44(2),
262--269.

\harvarditem{Belomestny and Rei\ss}{2006}{BR06} {\sc Belomestny, D.}
and {\sc Rei\ss, M.} (2006). Spectral calibration of exponential
L\'evy models. {\sl Finance Stoch.} {\bf 10}, 449--474.

\harvarditem{Cont and Tankov}{2004}{CT04} {\sc Cont, R.} and {\sc
Tankov, P.} (2004). {\sl Financial Modelling with Jump Processes}.
Chapman and Hall, Boca Raton.

\harvarditem{Chung}{1974}{Chu74} {\sc Chung, K.~L.} (1974). {\sl A
Course in Probability Theory. 2nd ed.} Academic Press, San Diego.

\harvarditem{Cs{\"o}rg{\H{o}} and Totik}{1983}{Cs83} {\sc
Cs{\"o}rg{\H{o}}, S.} and {\sc Totik, V.} (1983). On how long
interval is the empirical characteristic function uniformly
consistent. {\sl Acta Sci. Math. (Szeged)} {\bf 45}, 141--149.

\harvarditem{Dudley}{1989}{Dud89} {\sc Dudley, R.~M.} (1989). {\sl
Real Analysis and Probability}. Wadsworth, Belmont.

\harvarditem{Fan}{1991}{Fan91} {\sc Fan, J.} (1991). On the optimal
rates of convergence for nonparametric deconvolution problems. {\sl
Ann. Statist.} {\bf 19} 1257--1272.

\harvarditem{Feuerverger and McDunnough}{1981a}{FM81a} {\sc
Feuerverger, A.} and {\sc McDunnough, P.} (1981a). On the efficiency
of empirical characteristic function procedures. {\sl J. Royal
Statist. Soc.}, Ser.~B {\bf 43}, 20--27.

\harvarditem{Feuerverger and McDunnough}{1981b}{FM81b} {\sc Feuerverger, A.}
and {\sc McDunnough, P.} (1981b). On some Fourier methods for inference. {\sl
J. Amer. Statist. Assoc.} {\bf 76}, 379--387.

\harvarditem{Figueroa-L\'opez and Houdr\'e}{2006}{FH06} {\sc Figueroa-L\'opez,
J.} and {\sc Houdr\'e, C.} (2006). Risk bounds for the nonparametric estimation
of L\'evy processes in {\sl High Dimensional Probability}, IMS Lecture Notes
{\bf 51},  96--116.

\harvarditem{Gnedenko and Kolmogorov}{1968}{GK68} {\sc Gnedenko, B.
V.} and {\sc Kolmogorov, A. N.} (1968). {\sl Limit Distributions for
Sums of Independent Random Variables. 2nd ed.} Addison-Wesley,
Reading.

\harvarditem{Goldenshluger and Pereverzev}{2003}{GP03} {\sc
Goldenshluger, A.} and {\sc Pereverzev, S. V.} (2003). On adaptive
inverse estimation of linear functionals in Hilbert scales. {\sl
Bernoulli} {\bf 9}(5), 783--807.

\harvarditem{Gugushvili}{2007}{Gug07} {\sc Gugushvili, S.} (2007).
Decompounding under Gaussian noise. {\sl Math Archive} {\tt
arXiv:0711.0719v1}.

\harvarditem{Hall and Yao}{2003}{HY03} {\sc Hall, P.} and {\sc Yao,
Q.} (2003). Inference in components of variance models with low
replication. {\sl Ann. Statist.} {\bf 31}, 414--441.

\harvarditem{Jacod and Shiryaev}{2002}{JS02} {\sc Jacod, J.} and
{\sc Shiryaev, A.} (2002). {\sl Limit Theorems for Stochastic
Processes.} 2nd Edition. Grundlehren Vol. 288, Springer, Berlin.

\harvarditem{Jongbloed, van der Meulen, and van der
Vaart}{2005}{Hol05} {\sc Jongbloed, G., van der Meulen, F.H.} and
{\sc van der Vaart, A.W.} (2005). Nonparametric inference for
L\'evy-driven Ornstein-Uhlenbeck processes. {\sl Bernoulli} {\bf
11}(5), 759--791.

\harvarditem{Katznelson}{1976}{Ka76} {\sc Katznelson, Y.} (1976).
{\sl An Introduction to Harmonic Analysis}, 2nd Edition, Dover, New
York.

\harvarditem{Kolmogorov}{1932}{Kol32} {\sc Kolmogorov, A.N.} (1932). Sulla formula generale
di un processo stochastico omogeneo (Un problema di Bruno de Finetti) ({\em in Italian}), {\sl Rendiconti
della R. Accademia Nazionale dei Lincei (Ser. VI)} {\bf 15}, 866--869.

\harvarditem{Korostelev and Tsybakov}{1993}{KT93} {\sc Korostelev, A.P.} and
{\sc Tsybakov, A.B.} (1993). {\sl Minimax Theory of Image Reconstruction.}
Lecture Notes in Statistics 82, Springer, New York.

\harvarditem{K\"uchler and Tappe}{2008}{KT06} {\sc K\"uchler, U.}
and {\sc Tappe, S.} (2008). On the shapes of bilateral Gamma
densities, Stat. Proba. Letters, to appear.

\harvarditem{Mainardi and Rogosin}{2006}{MR06} {\sc Mainardi, F.} and {\sc Rogosin, S.} (2006).
The origin of infinitely-divisible distributions: from de Finetti's problem to L\'evy-Khinchine formula,
{\sl Mathematical Methods in Economics and Finance} {\bf 1}, 37--55.

\harvarditem{Nishiyama}{2007}{Nish07} {\sc Nishiyama, Y.} (2007).
Nonparametric estimation and testing time-homogeneity for processes
with independent increments. {\sl Stoch. Proc. Appl.}, to appear.

\harvarditem{van Es, Gugushvili, and Spreij}{2007}{EGS07} {\sc van
Es, B., Gugushvili, S.} and {\sc Spreij P.} (2007). A kernel-type
nonparametric density estimator for decompounding. {\sl Bernoulli}
{\bf 13}, 672--694.

\harvarditem{van der Vaart}{1998}{vdV98} {\sc van der Vaart, A.}
(1998). {\sl Asymptotic Statistics}. Cambridge University Press,
Cambridge.

\harvarditem{Watteel and Kulperger}{2003}{WK03} {\sc Watteel, R. N.}
and {\sc Kulperger, R. J.} (2003). Nonparametric estimation of the
canonical measure for infinitely divisible distributions. {\sl J.
Stat. Comput Simul.} {\bf 73}, 525--542.

\harvarditem{Yukich}{1985}{Yuk85} {\sc Yukich, J. E.} (1985). Weak
convergence of the empirical characteristic function. {\sl Proc.
Amer. Math. Soc.} {\bf 95}(3), 470--473.

\end{document}